\documentclass[11pt,a4paper,reqno]{amsart}
\usepackage{epsfig}
\usepackage{color}
\usepackage{amssymb,amsmath,amsthm,amstext,amsfonts}
\usepackage[normalem]{ulem}
\usepackage{amsmath,amscd,amstext,amsthm,amsfonts,latexsym}
\usepackage[colorlinks=true,linkcolor=blue, urlcolor=red, citecolor=blue, backref]{hyperref}	
\usepackage{amsmath,amssymb,amsthm}
\usepackage{color}
\usepackage[active]{srcltx}
\usepackage{amssymb, amsmath}
\usepackage{graphicx}
\usepackage{float}
\usepackage[active]{srcltx}

\usepackage[all,arc]{xy}
\usepackage{enumerate}
\usepackage{mathrsfs}

\newtheorem{theorem}{Theorem}[section]
\newtheorem{corollary}[theorem]{Corollary}
\newtheorem{proposition}[theorem]{Proposition}
\newtheorem{lemma}[theorem]{Lemma}

\theoremstyle{definition}
\newtheorem{maintheorem}{Theorem}

\newtheorem{maincorollary}[maintheorem]{Corollary}

\newtheorem{definition}[theorem]{Definition}

\newtheorem{example}[theorem]{Example}

\newtheorem{remark}[theorem]{Remark}

\newcommand{\htop}{h_{\topp}}

\newcommand{\vep}{\varepsilon}

\newcommand{\topp}{\operatorname{top}}

\newcommand{\Ptop}{P_{\topp}}

\newcommand{\cP}{\mathcal{P}}

\newcommand{\cE}{\mathcal{E}}

\newcommand{\cF}{\mathcal{F}}

\usepackage[normalem]{ulem}

\newcounter{main}

\makeatletter
\let\c@equation\c@theorem
\makeatother
\numberwithin{equation}{section}

\title[Entropy functions, variational principles and equilibrium states]{A Convex Analysis approach to Entropy functions, variational principles and equilibrium states}

\author[A. Bi\'s]{Andrzej Bi\'s}
\address{Faculty of Mathematics, {\L}\'od\'z University, Poland}
\email{andrzej.bis@wmii.uni.lodz.pl}

\author[M. Carvalho]{Maria Carvalho}
\address{CMUP \& Faculdade de Ci\^encias da Universidade do Porto, Portugal.}
\email{mpcarval@fc.up.pt}

\author[M. Mendes]{Miguel Mendes}
\address{CMUP \& Faculdade de Engenharia da Universidade do Porto, Portugal.}
\email{migmendx@fe.up.pt}

\author[P. Varandas]{Paulo Varandas}
\address{CMUP \& Departamento de Matem\'atica, Universidade Federal da Bahia, Salvador, Brazil.}
\email{paulo.varandas@ufba.br}

\keywords{Finitely additive set function; Pressure function; Variational principle; Equilibrium state; Phase transition; Transfer operator;
Sub-additive sequence; Linear cocycle; Lyapunov exponent; Singular value potential} 
\subjclass[2010]{Primary:
37D35, 
28D20, 
26A51; 
Secondary:
37D25, 
37C30, 
}

\date{\today}

\begin{document}

\begin{abstract}
Using methods from Convex Analysis, for each generalized pressure function we define an upper semi-continuous affine entropy-like map, establish an abstract variational principle for both countably and finitely additive probability measures and prove that equilibrium states always exist. We show that this conceptual approach imparts a new insight on dynamical systems without a measure with maximal entropy, may be used to detect second-order phase transitions, prompts the study of finitely additive ground states for non-uniformly hyperbolic transformations and grants the existence of finitely additive Lyapunov equilibrium states for singular value potentials generated by linear cocycles over continuous self-maps.
\end{abstract}

\maketitle

\section{Introduction}\label{se:introd}

\subsubsection*{\emph{\textbf{Classical thermodynamic formalism}}}

The modern theory of Dynamical Systems has its origins in the end of the nineteenth century with the pioneering work of Poincar\'e, who aimed at a complete description of the solutions of the differential equations modeling the three body problem in Celestial Mechanics. It is in the course of this investigation that Poincar\'{e} encounters the intriguing phenomenon of what was later named homoclinic tangencies. Several decades after this first surfacing of such complex dynamical behavior a common opinion had grown among researchers that geometric methods were insufficient to fully describe the asymptotic behavior of dynamical systems in general. In the meantime, some of the ideas from Statistical Mechanics (e.g. Boltzman ergodic hypothesis), dated from the previous century, were spreading around, and the main ergodic results of that period, namely von Neumann's and Birkhoff's ergodic theorems, were published. Yet, Kolmogorov's groundbreaking proposal made in the late 1950's of bringing both Probability and Entropy Theories into the realm of Dynamical Systems faced intrinsic difficulties, inasmuch as the latter domain of study had not yet been established as an independent domain of research.

The theory of uniformly hyperbolic dynamical systems, essentially born in the sixties and building over the existence of invariant foliations with exponential contracting and expanding behavior, brought forward both the models and the axiomatic framework which granted some understanding of the homoclinic behavior and support to the realization of Kolmogorov's insight. Meanwhile, the paradigmatic example of Smale's horseshoe made clear that subshifts of finite type could codify the hyperbolic dynamical systems, and that such coding was possible due to the existence of finite partitions with a Markov property. What's more, the finite Markov partitions for hyperbolic dynamical systems (both diffeomorphisms and flows) had a further benefit: they conveyed the thermodynamic formalism from Statistical Mechanics to Dynamical Systems, through keystone contributions by Sinai, Ruelle, Bowen, Ratner and Walters \cite{Bo, Bo-flows, Ra, Ru78, Si72, Wa}, among others.

The thermodynamic formalism for dynamical systems aims to prove the existence of invariant probability measures which maximize the topological pressure, besides reporting on their statistical properties. Such measures, called equilibrium states, are often Gibbs, and include as specific examples both probability measures absolutely continuous with respect to Lebesgue and measures of maximal entropy. For uniformly hyperbolic diffeomorphisms and flows, when restricted to a basic piece of the non-wandering set, equilibrium states exist and are unique for every H\"older continuous potential (cf.  \cite{Bo75,Ru78,Si72}). In the case of diffeomorphisms, the basic known strategy to prove this remarkable fact is to (semi-)conjugate the dynamics to a subshift of finite type, \emph{via} a finite Markov partition. Finer properties, including the analiticity of the pressure map and the relation between pressure, periodic orbits and dynamical zeta functions, were later addressed by Parry and Pollicott~\cite{PP90}.

\subsubsection*{\emph{\textbf{Non-uniform hyperbolicity and phase transitions}}}

The study of dynamical systems with weaker forms of hyperbolicity, such as partial hyperbolicity and dominated splittings, is still under development (cf. \cite{BDV}). In particular, an extension of the thermodynamic formalism beyond the scope of uniform hyperbolicity has been facing several intricacies. A very fruitful strategy to overcome some of them is the construction of induced and tower dynamics with hyperbolic behavior (cf. \cite{Pi,Sar13}). These induced maps can be well described by countable shifts, for which the thermodynamic formalism is more or less settled \cite{Sar15}, but introduces two new hindrances. On the one hand, not all invariant probability measures 
may be lifted to the induced tower dynamics (see \cite{BT} and references therein); on the other hand, equilibrium states for the tower dynamics may induce probability measures that are $\sigma$-finite on the phase space but this transfer process depends on the integrability of the return time function. While in the case of $C^\infty$-surface diffeomorphisms with positive entropy there are finitely many measures of maximal entropy (cf. \cite{BCS}), the previous two issues are among the reasons why the theory of thermodynamic formalism for non uniformly hyperbolic dynamical systems remains incomplete even for partially hyperbolic systems.

It is in this context that phase transitions (characterized either by the non-analiticity of the pressure map or by a discontinuity on the number of its equilibrium states) have been thoroughly studied. Phase transitions are reasonably well understood in the one-dimensional context (see \cite{CR19,IT,PR,PRS} and references therein) due to the recent tools to detect and characterize the sources of non-hyperbolic behavior. Beyond hyperbolicity even the existence of equilibrium states is far from being established, though the non-uniform hyperbolicity may be enough in special contexts, as happens in the case of Sinai billiard maps (cf. \cite{BD20}). Even so, Newhouse described in \cite{New89} sufficient conditions for the entropy map to be upper semi-continuous, which in particular ensure that $C^\infty$ maps of compact manifolds have equilibrium states for every continuous potential. In addition, a codification mechanism by symbolic extensions, and its existence for smooth dynamics, were also explored by Downarowicz and Newhouse in \cite{DN}. Nonetheless, one can ask under what general conditions do equilibrium states exist, seeing that there are $C^r$-surface diffeomorphisms, $1 \leqslant r < +\infty$, having no measure with maximal entropy (cf. \cite{Bu14}). One of the main goals in this paper is to prove that they always exist if one drops the requirement that they are countably additive.

\subsubsection*{\emph{\textbf{The use of Convex Analysis}}}

One of the fundamental mathematical tools used in Thermodynamics, Statistical Physics and Stochastic Finance is convexity. For instance, regarding the theory of lattice gases, Israel observes in \cite{Isr} that the pressure exists in the thermodynamic limit as a convex function in the space of interactions; and uses the pioneer work of Bishop and Phelps on Convex Analysis \cite{BiP} to construct equilibrium states, to prove that they satisfy Dobrushin-Lanford-Ruelle condition and to describe phase transitions. The aforesaid paper \cite{BiP} is also a cornerstone for several recent applications in Economics, namely on equilibrium theory, risk measures and stochastic finance (cf. \cite{FS-B} and references therein), where abstract variational principles appear associated to convex risk measures and penalty functions.

An important application of the Convex Analysis methods in the thermodynamic formalism within statistical mechanics is due to Israel and Phelps~\cite{IP} who, inspired by mathematical models in the classical theory of lattice gases \cite{Ru78}, investigated the differentiability, the tangent functionals and  the variational entropy (defined by means of the Legendre-Fenchel transform) in the case of generalized pressure functions acting on the space of affine real-valued continuous functions whose domain is a compact convex set. Apart from the natural convexity and continuity assumptions, and the fact that the pressure function acts on a set of affine maps, such pressure maps were assumed to also satisfy a strong positivity condition (cf. \cite[p.136]{IP}), a hypothesis which seldom occurs in the thermodynamic formalism context. Encouraged by the works of Israel and Phelps \cite{IP} and Lopes et al \cite{Lopes}, recent formulations of variational principles using Convex Analysis have appeared in  \cite{CER,CES,GKLM}, extending \cite{IP} to the dynamical context of shifts on Borel standard spaces and dynamical systems whose transfer operators have a spectral gap. Since the usual Kolmogorov-Sinai metric entropy is not adequate to the context of shifts on Borel standard spaces, these authors define a notion of entropy (which they also call variational entropy) for measures that are eigenvectors of a normalized transfer operator's dual, and prove, as a variational principle, that it coincides with the Legendre transform of the spectral radius of a suitable operator. Similar variational principles have been obtained in the context of weighted shift spaces acting on continuous or $L^1$-functions (see \cite{ABL2011,BaLe}). In both situations, it may happen that the pressure function, defined as the logarithm of the spectral radius of the corresponding transfer operator, even if well adapted to the transfer operator, is hardly related to the classical notion of topological pressure for dynamical systems.

In what follows, we will use Banach function spaces that contain non-continuous observable maps, which appear naturally in the context of piecewise smooth dynamical systems (see Subsection~\ref{sec:RPF-bounded}). We will focus on the space $B_m(X)$ of bounded measurable functions on a compact metric space, in which case the Convex Analysis dualities produce an entropy function on the space of finitely additive measures. In the seventies, Salomon Bochner wrote that, contrary to popular mathematical opinions, finitely additive measures were more interesting and perhaps more important than countable additive ones (see \cite{Mah}). Surely, finitely additive measures arise in Analysis, Statistics,  Measure Theory and Dynamical Systems (see \cite{BB, CER} and references therein), and are closely related to 
the Banach-Tarski paradox  (cf. \cite{Wag})
and the characterization of amenable groups (cf. \cite{Pat}). We will show in this paper that they can also be used to detect second-order phase transitions in the classical thermodynamic formalism (see Theorem~\ref{thm:main} and Subsection~\ref{sec:second-order}).

\subsubsection*{\emph{\textbf{Our main contributions}}}

We address the following problem: to find an abstract variational principle, valid for real-valued convex, monotone and translation invariant functions defined on a suitable Banach space of potentials, which is powerful enough to be applied to either the classical topological pressure for a continuous self-map or to the more recently defined notions of pressure for semigroup actions. Inspired by some recent applications of convex analysis to risk measures in stochastic finance \cite{FS-B} and by differentiability results on the classical topological pressure function \cite{Wa2}, we obtain a variational principle under great generality, which conveys an upper semi-continuous entropy-like map acting either on the space of finitely additive normalized set functions or on the space probability measures (cf. Theorem~\ref{thm:main}). Furthermore, we prove that these entropy-like functions, which are Fenchel-Legendre transforms of abstract pressure functions, are affine. This is a crucial property for both the study of differentiability properties of the pressure functions (see Theorem~\ref{thm:differentiable-functionals} and Corollary~\ref{thm:Gateaux}) and the characterization of the entropy-like function determined by the classical topological pressure, when restricted to the set of invariant probability measures, as the upper semi-continuous concave envelope of the measure-theoretic entropy (cf. Subsection~\ref{sse:star} for the precise definitions and Theorem~\ref{thm:second-main}).

In the special case of entropy functions defined over the space of probability measures, our contributions not only recover the classical variational principle for the topological pressure of continuous potentials but also improve the thermodynamic formalism of dynamics which do not admit equilibrium states, and allow us to address the more general family of bounded potentials (see Theorem~\ref{thm:second-main} for the precise statement). The latter turned out to be of special interest in order to consider the pressure functions linked to non-additive sequences of continuous potentials as singular value potentials associated to linear cocycles. Indeed, in this context we show that the pressure function of such non-additive sequences coincides with the pressure function of a single bounded observable, for which our results guarantee the existence of equilibrium states (see Theorem~\ref{cor:additive-2} and Subsection~\ref{subsec:cocycles} for more details).

The need to deal with finitely additive measures (respectively, probability measures) is a consequence of the duality theorem for bounded measurable functions (respectively, continuous functions with compact support) employed in the proof of Theorem~\ref{thm:main}. For finitely additive set functions there are known versions of Poincar\'e recurrence theorem \cite{BR}, Birkhoff ergodic theorem \cite{RamBET}, Central limit theorem \cite{Kar, RamCLT}, a  strong law of large numbers \cite{Chen1, Chen2} and a Radon-Nikodym theorem \cite{May}, besides information on the accumulation points of sequences of finitely additive set functions \cite{Nik}. Other useful properties of these set functions may be found in \cite{PS, Dub, Zub, CKK, Buf} and in the comprehensive book \cite{Tol}.

Finally, we observe that the abstract variational principle in Theorem~\ref{thm:main}  does not summon any dynamics, which makes it applicable in many contexts. In \cite{BCMV-2}, using several suitable notions of pressure, we have applied it to semigroup actions which are finitely generated by continuous maps, a setting which comprises free semigroups, countable sofic groups and groups endowed with a reference measure.

\section{Main results}\label{se:statements}

\subsection{Entropy functions for abstract pressure functions}

Let $(X,d)$ be a metric space and let $\mathfrak{B}$ stand for the $\sigma$-algebra of Borel subsets of $X$. Denote by ${\mathbf B}$ a Banach space over the field $\mathbb{R}$ equal to either
\begin{eqnarray}\label{eq:space-B}
B_m(X) &=& \big\{\varphi: \, X \,\to\, \mathbb{R} \,\,|\,\, \varphi \text{ is Borel measurable and bounded}\big\}\\
\text{ or } \,\,C_b(X) &=&  \big\{\varphi \in B_m(X) \,\,|\,\, \varphi \text{ is continuous}\big\}\nonumber\\
\text{ or else } \,\,C_c(X) &=& \big\{\varphi \in C_b(X) \,\,|\,\, \varphi \text{ has compact support}\big\} \nonumber
\end{eqnarray}
endowed with the norm $\|\varphi\|_\infty = \sup_{x \, \in \, X}\, |\varphi(x)|$. In what follows, $\mathcal{P}_a(X)$ will stand for the set of $\mathfrak{B}$-measurable, regular and normalized finitely additive set functions on
$X$, 
which we will simply call \emph{finitely additive probability measures}, with the total variation norm (cf. \cite[IV.2.15]{DS}). Recall that the total variation norm in $\cP_a(X)$ is given by
$$\|\mu - \nu\| = \sup\, \left\{\Big|\int \psi \, d\mu - \int \psi \, d\nu\Big| \,\colon \,\,\,\psi \in \mathbf B \,\text{ and }\, \|\psi\|_\infty \leqslant 1\right\}.$$
As the set of regular, finitely additive measures endowed with the total variation norm is a Banach space isometrically isomorphic to the topological dual of $C_b(X)$ (cf. \cite[Theorem~14.9]{AB06}),
the set $\mathcal{P}_a(X)$ is compact in the weak$^*$ topology as a consequence of the classical Banach-Alaoglu theorem.
For future use, we denote by $\cP(X) \subset \mathcal{P}_a(X)$ the set of Borel $\sigma$-additive probability measures on $X$ endowed with the weak$^*$ topology, and by $C(X)$ the space of real valued continuous maps whose domain is $X$.
Inspired by \cite{FS-B}, we introduce the following axiomatic definition of pressure function.

\begin{definition}\label{def:pressure-function}
A function $\Gamma: {\mathbf B} \to \mathbb{R}$ is called a \emph{pressure function} if it satisfies the following conditions:
\begin{enumerate}
\item[(C$_1$)] \emph{Monotonicity}: $\varphi \leqslant \psi \,\, \Rightarrow \,\, \Gamma(\varphi) \leqslant \Gamma(\psi) \quad \forall\,\varphi, \, \psi \in {\mathbf B}$.
\medskip
\item[(C$_2$)] \emph{Translation invariance}: $\Gamma(\varphi + c) = \Gamma(\varphi) + c \quad \forall \,\varphi \in {\mathbf B} \quad \forall \,c \in \mathbb{R}$.
\medskip
\item[(C$_3$)] \emph{Convexity}: $\Gamma(t \,\varphi + (1-t) \,\psi) \leqslant t\, \Gamma(\varphi) + (1-t)\, \Gamma(\psi) \quad \forall \,\varphi, \, \psi \in {\mathbf B} \quad \forall\, t \in \,[0,1]$.
\end{enumerate}
\end{definition}

We note that, in comparison to endomorphisms, pressure functions and the abstract framework of \cite{IP}, here we do not require the pressure function to preserve co-boundary type functions. This is crucial in \cite{BCMV-2}, where we deal with very abstract pressure functions associated to group and semigroup actions.

It is not hard to check that properties (C$_1$) and (C$_2$) imply that any pressure function is Lipschitz continuous, meaning that
$| \Gamma(\varphi)  - \Gamma(\psi) | \leqslant \|\varphi-\psi\|_\infty$ for every $\varphi,\, \psi\in {\mathbf B}$.

Our first result, inspired by \cite{FS-B}, establishes a general variational principle and the existence of finitely additive equilibrium states for any pressure function.

\begin{maintheorem}\label{thm:main} \emph{Let $\Gamma: {\mathbf B} \to \mathbb{R}$ be a pressure function. Then
\begin{equation}\label{eq:var}
\Gamma(\varphi) = \max_{\mu \, \in \, \mathcal{P}_a(X)}\, \left\{{\mathfrak h}(\mu) + \int \varphi \, d\mu \right\} \quad \quad \forall\, \varphi \in {\mathbf B}
\end{equation}
where
\begin{equation}\label{eq:def-cone}
{\mathfrak h}(\mu) = \inf_{\varphi \, \in \, \mathcal{A}_\Gamma}\, \left\{\int  \varphi \, d\mu\right\} \quad \quad \text{and} \quad \quad \mathcal{A}_\Gamma = \big\{\varphi \in {\mathbf B} :\, \Gamma(-\varphi) \leqslant 0 \big\}.
\end{equation}
Moreover, ${\mathfrak h}(\mu)$ is affine and upper semi-continuous; and if $\alpha: \mathcal{P}_a(X) \to \mathbb{R} \cup \{-\infty,\, +\infty\}$ is another function taking the role of ${\mathfrak h}$ in \eqref{eq:var}, then $\alpha \leqslant {\mathfrak h}$. In addition, one has
$${\mathfrak h}(\mu) = \inf_{\varphi \, \in \, {\mathbf B}}\, \left\{\Gamma(\varphi) - \int\varphi \, d\mu\right\} \quad \quad \forall \,\mu \in  \mathcal{P}_a(X).$$
If $X$ is locally compact and ${\mathbf B}=C_c(X)$ then the maximum in ~\eqref{eq:var} is attained in $\mathcal P(X)$.}
\end{maintheorem}

Some comments are in order. Theorem~\ref{thm:main} can be understood as a variant of the Fenchel-Moreau Theorem in Convex Analysis (see Theorem VI.5.3 in \cite{Ellis}). We stress that we do not request compactness of the space $X$, just local compactness when considering ${\mathbf B}=C_c(X)$ due to the use of the Riesz-Markov Theorem, and that the supremum in \eqref{eq:var} can be computed taking only extremal measures (cf. \eqref{eq:vp-extreme}). We also highlight the fact that \eqref{eq:def-cone} asserts that the function ${\mathfrak h}$ is computed by averaging potentials whose additive inverses have non-positive pressure.
Observe, in addition, that
${\mathfrak h}(\mu)\geqslant 0$ if and only if $\Gamma(\varphi) \geqslant \int\varphi \, d\mu \quad \forall\, \mu \in  \mathcal{P}_a(X)$, a condition which extends a characterization of invariant probability measures for maps with upper semi-continuous Kolmogorov-Sinai entropy (cf. \cite[Theorems 9.11 \& 9.12]{Wa}).

\begin{remark}\label{rmk:relation-pressures}
The entropy function $\mathfrak h=\mathfrak h_{\Gamma,\mathbf B}$ depends on both the pressure function $\Gamma$ and its domain. For instance, if $X$ is a compact metric space and
$\Gamma: B_m(X) \to \mathbb R$ is a pressure function then $\Gamma\left|_{C(X)} \right.: C(X) \to\mathbb R$ is a pressure function as well, though with a different domain; and, according to ~\eqref{eq:def-cone}, the corresponding functions $\mathfrak h$ are computed respectively by
$$\mathfrak h_{B_m(X)}(\mu)
= \inf \left\{ \int\varphi \, d\mu \colon \, \varphi \, \in \,B_m(X)\, \;\text{and}\;
 \Gamma(-\varphi) \leqslant 0\right\} \quad \forall \mu\in \cP_a(X)$$
and
$$\mathfrak h_{C(X)}(\mu)
= \inf \left\{\int\varphi \, d\mu \colon \, \varphi \, \in \,C(X)\, \;\text{and}\;
 \Gamma(-\varphi) \leqslant 0\right\}
 \qquad \forall \mu\in \cP(X).$$
In particular, as $\mathfrak h_{B_m(X)}$ is computed on the larger space of finitely additive measures, one has that $\mathfrak h_{B_m(X)}(\mu)  \;\leqslant\;   \mathfrak h_{C(X)}(\mu)$ for every $\mu\in \cP(X)$.
\end{remark}

\begin{remark}
It is straightforward to conclude that $\mu_0$ attains the minimum in \eqref{eq:def-cone} (that is, $\mathfrak{h}({\mu_0}) = \int \varphi_0 \, d\mu_0$ for some $\varphi_0 \in \mathcal{A}_\Gamma$) if and only if $\Gamma(-\varphi_0) = 0$ and $\mu_0$ is an equilibrium state of $-\varphi_0$.  For instance, the minimum is attained when $\mu_0$ is a measure of maximal entropy (in which case $\varphi_0 = \htop(f)$); or when $\mu_0$ is the Sinai-Ruelle-Bowen measure of a $C^2$ Axiom A attractor (in which case $\varphi_0 = -\log \text{Jac}(Df_{|E^u})$; cf. \cite{Bo75}).

\end{remark}

\subsection{Uniqueness of the equilibrium states}

The variational principle stated in Theorem~\ref{thm:main} ensures that there always exist normalized finitely additive measures for which the right-hand side of \eqref{eq:var} attains the supremum; that is, the set
$\cE_\varphi(\Gamma) =\Big\{\mu\in \cP_a(X) \colon \, \Gamma(\varphi)={\mathfrak h}(\mu) + \int \varphi \, d\mu  \Big\}$
is non-empty. This raises the subtle question of whether they are unique. To address this issue we consider functionals tangent to the pressure in our abstract framework.

Given a pressure function $\Gamma: {\mathbf B} \to \mathbb{R}$ and a potential $\varphi \in {\mathbf B}$, we say $\mu \in \cP_a(X)$ is a \emph{tangent functional to $\Gamma$ at $\varphi$} (also known as \emph{sub-differential}) if
\begin{equation}\label{eq:tf1}
\Gamma(\varphi+\psi) -\Gamma(\varphi) \geqslant \int \psi \, d\mu \qquad \forall \, \psi \in {\mathbf B}.
\end{equation}
As in \cite{Wa2}, the continuity of $\Gamma$ and the Hahn-Banach Theorem guarantee that the space $\mathcal T_\varphi(\Gamma)$
of tangent functionals to $\Gamma$ at $\varphi$ is non-empty for every $\varphi$; and it is easily seen to be a convex and weak$^*$ compact set. The use of the Hahn-Banach theorem makes that argument for the existence of such measures is not constructive. The next result states that the space of tangent functionals to $\Gamma$ at $\varphi \in {\mathbf B}$ coincides with the space of finitely additive probability measures attaining the maximum on \eqref{eq:var}, and that typical
potentials have a unique equilibrium state. More precisely:

\begin{maintheorem}\label{thm:tangent-functionals}
\emph{Let $\Gamma: {\mathbf B} \to \mathbb{R}$ be a {pressure function}. Then
$$\cE_\varphi(\Gamma)=\mathcal T_\varphi(\Gamma) \quad \quad \forall\,\varphi\in {\mathbf B}.$$
Moreover, if $\mathbf B=C_b(X)$ or $\mathbf  B=C_c(X)$, then there exists a residual subset $\mathfrak R \subset \mathbf B$ such that $\#\,\cE_\varphi(\Gamma)=1$ for every $\varphi\in \mathfrak R$.}
\end{maintheorem}

The previous result extends the work of Walters \cite{Wa2} within the context of the classical topological pressure, for which the uniqueness of equilibrium states is tied in with the differentiability of the pressure.

A pressure function $\Gamma: \mathbf B \to \mathbb R$ is said to be \emph{locally affine at $\varphi \in \mathbf B$} if there exist a neighborhood $\mathcal{V}$ of $0$ in $\mathbf B$ and a unique $\mu_\varphi \in \cP_a(X)$ such that
\begin{equation}\label{eq:FD}
\Gamma(\varphi+\psi)-\Gamma(\varphi)=\int \psi \; d\mu_\varphi \quad \quad \forall \, \psi \in \mathcal{V}.
\end{equation}
In particular, $\mathcal T_\varphi(\Gamma) = \{\mu_\varphi\}$. If $\Gamma$ is locally affine at all elements of $\mathbf B$, then $\mu_\varphi$ does not depend on $\varphi$, and $\Gamma$ is said to be \emph{affine}.
We say that the pressure function
 $\Gamma: \mathbf B \to \mathbb R$ is \emph{Fr\'echet differentiable} at $\varphi \in \mathbf B$ if there exists a unique $\mu_\varphi \in \cP_a(X)$ such that
\begin{equation}\label{eq:Frechet}
\lim_{\psi \,\to \,0}\, \frac1{\|\psi\|_\infty}\,\, {\Big|\Gamma(\varphi+\psi)-\Gamma(\varphi) -\int \psi \; d\mu_\varphi\Big|}=0.
\end{equation}
It is known (cf. \cite[Theorem 6]{Wa2}) that the local affine property is equivalent to Fr\'echet differentiability when one considers the classical topological pressure function. We will show that the same statement for general pressure functions is also true.

\begin{maintheorem}\label{thm:differentiable-functionals}
\emph{Let $\Gamma: {\mathbf B} \to \mathbb{R}$ be a {pressure function}. The following assertions are equivalent:
\begin{enumerate}
\item[\emph{(a)}] $\Gamma$ is locally affine at $\varphi$.
\item[\emph{(b)}] There exists a unique tangent functional in $\mathcal T_\varphi(\Gamma)$ and
$$\lim_{\psi \,\to \,0} \,\sup\, \Big\{\|\mu - \mu_\varphi\| \colon \,\,\mu \in \mathcal T_{\varphi+\psi}(\Gamma)\Big\}= 0.$$
\item[\emph{(c)}] $\Gamma$ is Fr\'echet differentiable at $\varphi$.
\end{enumerate}}
\noindent \emph{Therefore, the following statements are mutually equivalent as well: $(\bar{a})$ $\Gamma$ is affine, $(\bar{b})$
$\bigcup_{\varphi\,\in\,\mathbf B} \,\mathcal T_\varphi(\Gamma)$ is a singleton, and $(\bar{c})$
$\Gamma$ is everywhere Fr\'echet differentiable.}
\end{maintheorem}

As being affine is a rigid condition, the previous theorem also conveys the information that a pressure function $\Gamma$ is rarely everywhere Fr\'echet differentiable. Consequently, typical pressure functions either exhibit more than one tangent functional at some element of $\mathbf B$, or these do not vary continuously in the operator norm (see \cite{DE} for examples).
The previous discussion prompts us to consider the weaker notion of Gateaux differentiability. A pressure function $\Gamma: \mathbf B \to \mathbb R$ is said to be \emph{Gateaux differentiable at $\varphi \in \mathbf B$} if, for every $\psi \in \mathbf B$ the directional pressure map $t \in \mathbb R  \,\mapsto \,\Gamma(\varphi + t\psi)$ is differentiable, that is,
the limit
$$d\,\Gamma(\varphi)(\psi) \,:= \,\lim_{t\,\to \,0} \,\frac{1}{t} \,[\,\Gamma(\varphi + t\psi)-\Gamma(\varphi)\,]$$
exists and is finite for every $\psi \in \mathbf B$. Concerning real valued convex functions on Banach spaces, Walters proved in \cite[Corollary 2]{Wa2} a criterion for Gateaux differentiability.
In our setting, the corresponding statement reads as follows.

\begin{maincorollary}\label{thm:Gateaux}
\emph{Let $\Gamma: \mathbf B \to \mathbb{R}$ be a {pressure function}. Then $\Gamma$ is Gateaux differentiable at $\varphi$ if and only if  there exists a unique tangent functional in $\mathcal T_\varphi(\Gamma)$.}
\end{maincorollary}

\subsection{Existence of equilibrium states}

The duality results in Theorem~\ref{thm:main} can also be used to obtain a new insight on the thermodynamic formalism of a continuous self-map on a compact metric space. Indeed, in this context we provide the following variational principle regarding the classical topological pressure (we refer the reader to Subsection~\ref{sse:theorem 5} for the definitions).

\begin{maintheorem}\label{thm:second-main}
\emph{Let $f: \,X \to X$ be a continuous transformation of a compact metric space $X$ with $\htop(f) < +\infty$. The upper semi-continuous entropy map ${\mathfrak h}_f:\,\mathcal{P}(X) \, \to \, \mathbb{R}$ given by
\begin{equation*}\label{eq:second-dual}
{\mathfrak h}_\mu(f) = \inf_{\varphi \, \in \, C(X)}\, \left\{\Ptop(f, \varphi) - \int \varphi \, d\mu\right\}\quad \quad \forall \, \mu \in \mathcal{P}(X)
\end{equation*}
satisfies:
\begin{enumerate}
\item[\emph{(a)}] $0 \,\,\leqslant \,\,  h_\mu(f) \,\,\leqslant \,\, 
{\mathfrak h}_\mu(f) \quad \forall \, \mu \in \mathcal{P}_f(X).$
\smallskip
\item[\emph{(b)}] For every continuous potential $\varphi: X\to\mathbb R$,
\begin{equation}\label{eq:newVPf}
\Ptop(f,\varphi) = \max_{\mu \,\in\, \mathcal{P}(X)} \,\Big\{{\mathfrak h}_\mu(f) + \int \varphi \, d\mu \Big\}	
\,=\, \max_{\mu \,\in \,\mathcal{P}_f(X)}\,\left\{{\mathfrak h}_\mu(f) + \int \varphi \, d\mu\right\}
\end{equation}
where $\mathcal{P}_f(X)$ denotes the space of $f$-invariant Borel probability measures on $X$ endowed with the weak$^*$ topology.
\smallskip
\item[\emph{(c)}] Every measure $\mu\in \cP(X)$ which attains the maximum \eqref{eq:newVPf} is $f$-invariant.
\smallskip
\item[\emph{(d)}] The restriction of ${\mathfrak h}_f$ to $\mathcal{P}_f(X)$ is the upper semi-continuous concave envelope of Kolmogorov-Sinai metric entropy.
\end{enumerate}}
\end{maintheorem}

The elements of the set
\begin{equation}\label{def:D}
\mathscr D \,=\, \big\{ \mu\in \cP_f(X) \colon 0\leqslant h_\mu(f) < {\mathfrak h}_\mu(f) \big\}
\end{equation}
are the $f$-invariant probability measures where the entropy function $\mu \in  \mathcal{P}_f(X) \,\mapsto \, h_\mu(f)$ fails to be upper semi-continuous. The classical variational principle together with Theorem~\ref{thm:second-main} (b) imply that
$$\sup_{\mu \,\in \,\mathcal{P}_f(X)}\,\left\{h_\mu(f) + \int \varphi \, d\mu\right\}
	\;\;=\;\; \sup_{\mu \,\in \,\mathcal{P}_f(X)}\,\left\{{\mathfrak h}_\mu(f) + \int \varphi \, d\mu\right\}$$
even if $\mathscr D\neq\emptyset$. This shows that the measures in $\mathscr D$ (which play a role for instance if the entropy spectrum $\{h_\mu(f)\colon \mu\in \cP_f(X)\}$ is not an interval) do not affect the supremum. Still, such measures have a say when we look for equilibrium states (see an example in Subsection~\ref{sec:non-standard}). More details on the relation between the star-entropy and the $\mathfrak h$-entropy are given in Proposition~\ref{prop:max-gamma}. Finally, it is worthwhile mentioning that one can obtain a similar statement to the one of Theorem~\ref{thm:second-main} using finitely additive measures instead of probability measures (see Remark~\ref{rmk:otherbound}).

\begin{remark}
When $\htop(f) < + \infty$ and, additionally, the entropy map $\mu \in  \mathcal{P}_f(X) \,\mapsto \, h_\mu(f)$ is upper semi-continuous (as happens, for instance, when $f$ is expansive), Theorem~\ref{thm:second-main} and \cite[Theorem 9.12]{Wa} imply that
${\mathfrak h}_\mu(f) = h_\mu(f)$ for every $\mu \in \mathcal{P}_f(X).$
Thus, under these assumptions, the map ${\mathfrak h}_f: \mathcal{P}(X) \to \mathbb R$ is an extension of the metric entropy (with negative values at $\mathcal{P}(X) \setminus \mathcal{P}_f(X)$).
\end{remark}

The previous results pave the way to the description of multifractal analysis for Birkhoff averages, large deviations or ergodic optimization in both hyperbolic and non-hyperbolic contexts. Actually, while upper semi-continuity of entropy and uniqueness of equilibrium states are useful ingredients to provide a full description of the entropy map, the dimension of the level sets associated to Birkhoff averages and the maximizing probability measures, these properties may fail beyond the realm of uniform hyperbolicity. We will give a simple illustration through an application in ergodic optimization (see Subsection~\ref{sec:erg-opt} for definitions and the concept of zero temperature limits).
Taking into account Theorem~\ref{thm:second-main} 
and \cite[Theorem~4.1]{Je} (replacing the usual entropy function by the upper semi-continuous map $\mu \mapsto {\mathfrak h}_\mu(f)$) we obtain the following information.

\begin{maincorollary}\label{thm:zerotemp}
\emph{Let $f: X \to X$ be a continuous map on a compact metric space $X$ such that $\htop(f)<+\infty$ and let $\varphi \colon X \to\mathbb R$ be a continuous observable. Then
$$\frac1t \Ptop(f,t\varphi) =\max_{\nu\,\in \,\cP_f(X)}\, \Big\{\frac1t \,{\mathfrak h}_\nu(f) + \int \varphi \, d\nu  \Big\} \quad \underset{t\,\to\, +\infty}{\longrightarrow} \quad \max_{\nu\,\in\, \cP_f(X)} \,\int \varphi\, d\nu.$$
Moreover, if for each $t>0$ one chooses an $f$-invariant probability measure $\nu_t \in \cP_f(X)$ satisfying $\Ptop(f,t\varphi) = {\mathfrak h}_{\nu_t}(f) + \int t\, \varphi \, d\nu_t$ then
any weak$^*$ accumulation point $\nu_\infty \in \cP_f(X)$ of $(\nu_t)_{t\, >\, 0}$ as $t\to+\infty$ satisfies:
\smallskip
\begin{enumerate}
\item[\emph{(a)}] $\int \varphi \, d\nu_\infty = \underset{\mu\,\in\, \cP_f(X)}{\max} \, \int \varphi\, d\mu$
\medskip
\item[\emph{(b)}]${\mathfrak h}_{\nu_\infty}(f)
	= \underset{\nu\,\in\, \mathcal{M}_f(X, \,\varphi)}{\max} \,\,{\mathfrak h}_\nu(f) \,\,
	\geqslant\,\, \underset{\nu\,\in\,\, \mathcal{M}_f(X, \,\varphi)}{\max}\, h_\nu(f)$
\medskip
\item[\emph{(c)}] $\underset{t\,\to\, +\infty}{\lim}\,\, {\mathfrak h}_{\nu_t}(f) = {\mathfrak h}_{\nu_\infty}(f)$
\end{enumerate}
where $\mathcal{M}_f(X, \varphi)$ stands for the set of $f$-invariant $\varphi$-maximizing probability measures obtained through zero-temperature limits.}
\end{maincorollary}

\subsection{Non-additive thermodynamic formalism}

The previous Convex Analysis conclusions also apply to the thermodynamic formalism and the ergodic optimization for sub-additive sequences of continuous potentials, which include as special cases the singular value potentials
associated to linear cocycles,
a topic that goes beyond the scope of hyperbolic dynamical systems (see Section~\ref{se:thermo-formalism-0} for definitions and a discussion on previous results). This is so in spite of the fact that the Legendre-Fenchel duality is not available in this context, since the space of sub-additive sequences of continuous functions is not a vector space.
As far as we know, this approach brings forth novelty even in the case when $f$ is the one-sided full shift.

In what follows, given a cocycle $A \in C(X, GL(\ell,\mathbb R))$ over a continuous map $f: X\to X$, we denote by $P(f, A,\,\Phi_{\vec\alpha})$ the topological pressure of its singular value potential $\Phi_{\vec\alpha}$ (see Definition~\ref{def:press}, equation~\eqref{eq:Old-VP}
and Subsection~\ref{subsec:lcocs} for the precise definitions).

\begin{maintheorem}\label{thm:non-additive}
\emph{Let $f$ be a continuous self-map of a compact metric space $(X,d)$. There is an upper semi-continuous map ${\mathfrak h}_f: \mathcal{P}_a(X) \to \mathbb{R}$ such that, for every cocycle $A \in C(X, GL(\ell,\mathbb R))$, every vector $\vec \alpha=(\alpha_1, \alpha_2, \dots, \alpha_\ell) \in \mathbb R^\ell$ with $\alpha_1\geqslant \alpha_2 \geqslant \dots \geqslant \alpha_\ell$ and the corresponding non-additive sequence $\Phi_{\vec\alpha}$ of singular value potentials, there exists a map $\psi_{\Phi_{\vec\alpha}} \in B_m(X)$ such that
$$P(f, A,\,\Phi_{\vec\alpha}) \,= \max_{\mu \,\in \,\mathcal P_a(X)}\,\Big\{{\mathfrak h}_\mu(f) + \int  \psi_{\Phi_{\vec\alpha}}\,d\mu \Big\}$$
and
$$\forall\, \mu \in \mathcal{P}_f(X) \quad \quad \psi_{\Phi_{\vec\alpha}}(x) \,\,=\,\, \sum_{i=1}^k \alpha_i \cdot \lambda_i(A,x) \quad \quad \text{at $\mu$-almost every $x \in X$}.$$
Moreover, the set of finitely additive equilibrium states is non-empty for every linear cocycle in $C(X, GL(\ell,\mathbb R))$ and the zero temperature limits of finitely additive equilibrium states have the largest value of ${\mathfrak h}_f$ amongst the Lyapunov optimizing measures.}
\end{maintheorem}

\section{A variational principle: Proof of Theorem~\ref{thm:main}}\label{se:press-functional}

We begin this section by recalling some duality results on the Banach spaces we will consider.
If $X$ is locally compact, it follows from the Riesz-Markov representation Theorem that the dual of $C(X)$ can be identified with the collection of all finite signed measures on $(X, \mathfrak{B})$ equipped with the weak$^*$ topology, and that the subset of its positive normalized elements corresponds to the space $\cP(X)$
(cf. \cite[pp. 253]{KT}). This is clearly the case if $X$ is a compact metric space, where
$C(X)=C_c(X)=C_b(X)$. In the non-compact setting, the dual of $C_b(X)$ is identified with the space of regular
Borel and finitely additive measures endowed with the total variation norm (cf. Theorem~14.9 in \cite{AB06}).
An extension of the Riesz-Markov Theorem informs that the dual of $B_m(X)$ is represented by the space of Borel finitely additive measures with the topology induced by the total variation norm
(see \cite{FK,Hi}), whose subset of positive normalized elements corresponds to $\cP_a(X)$.

For the sake of completeness and rigor, we include a proof of the first part of Theorem~\ref{thm:main} along the lines of that provided by F\"ollmer and Schied (see \cite{FS-B}). 
Let $(X,d)$ be a metric space, ${\mathbf B} = B_m(X)$ and $\Gamma: {\mathbf B} \to \mathbb{R}$ be a pressure function. Define
$$ \mathcal{A}_\Gamma = \big\{\varphi \in {\mathbf B} :\, \Gamma(-\varphi) \leqslant 0 \big\}\quad \quad \text{and} \quad \quad {\mathfrak h}(\mu) \, := \, \inf_{\varphi \, \in \, \mathcal{A}_\Gamma}\, \left\{\int  \varphi \, d\mu\right\}.$$
We start showing that, for every $\varphi \in {\mathbf B}$, one has
\begin{equation}\label{eq:geq}
\Gamma(\varphi) \,\,\geqslant \,\,\sup_{\mu \, \in \, \mathcal{P}_a(X)}\, \left\{{\mathfrak h}(\mu) + \int \varphi \, d\mu \right\}.
\end{equation}
Due to translation invariance it is clear that $\Gamma(\varphi-\Gamma(\varphi))=0$. Thus $\widetilde{\varphi}:= \Gamma(\varphi)-\varphi$ belongs to $\mathcal{A}_\Gamma$. Therefore, for every $\mu \in \mathcal{P}_a(X)$
$${\mathfrak h}(\mu) \,\,\leqslant\,\, \int \widetilde{\varphi} \ d\mu = \Gamma(\varphi)-\int \varphi \ d\mu$$
which implies \eqref{eq:geq}. Conversely, given $\varphi \in {\mathbf B}$ we need to find $\mu_{\varphi} \in \mathcal{P}_a(X)$ such that
\begin{equation}\label{eq:leq}
\Gamma(\varphi) \,\,\leqslant\,\,  {\mathfrak h}(\mu_{\varphi}) + \int \varphi \ d\mu_{\varphi}.
\end{equation}
Yet, it is enough to do it for $\varphi$ such that $\Gamma(\varphi) = 0$ since the general case follows from taking
$\varphi-\Gamma(\varphi)$ and the translation invariance property of $\Gamma$. So, consider $\varphi \in {\mathbf B}$ such that $\Gamma(\varphi) = 0$. Thus the observable $-\varphi$ does not belong to the set
$$\mathcal{B}_{\Gamma}\,= \,\,\left\lbrace \psi \in {\mathbf B}:\, \Gamma(-\psi) < 0 \right\rbrace$$
which is convex and open due to the convexity and continuity of $\Gamma$. Therefore, by the geometric version of the Hahn-Banach Theorem there is a continuous, not identically zero, linear functional
$L: \,{\mathbf B} \, \to \, \mathbb{R}$ which separates the sets $\left\lbrace -\varphi \right\rbrace$ and $\mathcal{B}_\Gamma$ in the sense that
\begin{equation}\label{eq:LL}
L(-\varphi) \,\,\leqslant\,\, \inf_{\psi \,\in \,\mathcal{B}_\Gamma} \, L(\psi).
\end{equation}
By linearity of $L$, this is equivalent to saying that
\begin{equation}\label{eq:L}
L(\varphi) + \inf_{\psi \,\in \,\mathcal{B}_\Gamma} \, L(\psi) \,\geqslant\, 0.
\end{equation}

\begin{lemma}\label{le:positive} $L$ is positive and $L(1) > 0$, where $1$ stands for the constant function equal to one.
\end{lemma}

\begin{proof}
Consider $\psi \in {\mathbf B}$ such that $\psi \geqslant 0$. Firstly, let us see that for every $\lambda > 0$ we have $\lambda \psi + c \in \mathcal{B}_\Gamma$, where $c=\Gamma(0) + 1$. Indeed, by translation invariance and monotonicity of $\Gamma$ one has
$$\Gamma(-\lambda \psi - c) \,=\,\Gamma(-\lambda \psi) - c \,\,\leqslant\,\,  \Gamma(0) - c  < 0.$$
Due to \eqref{eq:LL}, this in turn implies that
$$L(-\varphi) \,\,\leqslant\,\, L(\lambda \psi + c) \,=\, \lambda L(\psi) + c L(1) \quad \quad \forall \,\lambda>0.$$
Thus, if $L(\psi) < 0$ then $L(-\varphi) = -\infty$, leading to a contradiction
with the fact that $L$ is a bounded functional.
This proves that $L$ is a positive functional. In particular, $L(1) \geqslant 0$.

Let us now prove that $L(1) \neq 0$. As $L$ is linear and not identically zero, we may take $\psi_0 \in {\mathbf B}$ such that $L(\psi_0) > 0$ and $\left\| \psi_0 \right\|_\infty < 1.$ Write $\psi_0 = \psi_0^{+} - \psi_0^{-}$, where $\psi_0^{+} = \max \,\{\psi_0, 0\}$ and $\psi_0^{-} = \max \,\{-\psi_0, 0\}$. Then, as $L$ is positive, we have
$$L(\psi_0^{+}) \,=\, L(\psi_0) + L(\psi_0^{-}) \,\,\geqslant\,\, L(\psi_0) > 0 \quad \quad \text{and} \quad \quad L(1 -\psi_0^{+}) \geqslant 0$$
since $1-\psi_0^{+} \geqslant 0$. Using again both the linearity and the monotonicity of $L$ we finally conclude that
$$L(1) \,=\,  L(1 -\psi_0^{+}) + L(\psi_0^{+}) \,\,\geqslant\,\, L(\psi_0^{+}) > 0.$$
\end{proof}

The previous lemma indicates that the continuous linear operator $\frac{L}{L(1)}$ is positive and normalized.
Therefore, according to an extension of the Riesz-Markov Representation Theorem \cite{DS}, there is a finitely additive probability measure $\mu_{\varphi} \in \mathcal{P}_a(X)$ (which belongs to $\mathcal{P}(X)$ if $\mathbf{B}=C_c(X)$ and $X$ is locally compact) such that
$$\int \psi \ d\mu_{\varphi} \,=\, \frac{L(\psi)}{L(1)} \quad \quad \forall \,\psi \in {\mathbf B}.$$
We are left to show that $\mu_\varphi$ satisfies \eqref{eq:leq}.
Observe that for every $\psi \in \mathcal{A}_{\Gamma}$ and every $\varepsilon> 0$ we have that $\psi + \varepsilon \in \mathcal{B}_\Gamma$. In other words,
$\mathcal{A}_{\Gamma,\,\varepsilon}\,:= \,\,\left\lbrace \psi + \varepsilon: \psi \in \mathcal{A}_\Gamma \right\rbrace \,\subset\, \mathcal{B}_\Gamma.$
Hence
$${\mathfrak h}(\mu_{\varphi}) \,=\, \inf_{\psi \,\in\, \mathcal{A}_{\Gamma}}\, \int \psi \, d\mu_{\varphi} \,
\leqslant\,\, \inf_{\psi \,\in\, \mathcal{B}_\Gamma}\, \int \psi \, d\mu_{\varphi} \,
\,\leqslant\,\, \inf_{\psi \,\in\, \mathcal{A}_{\Gamma,\,\varepsilon}}\, \int \psi \, d\mu_{\varphi} \,=\,  \inf_{\psi \,\in\, \mathcal{A}_{\Gamma}}\, \int \psi \, d\mu_{\varphi} + \varepsilon.$$
Since $\varepsilon>0$ is arbitrary we conclude that
${\mathfrak h}(\mu_{\varphi}) \,=\, \inf_{\psi \,\in \,\mathcal{B}_\Gamma}\, \int \psi \ d\mu_{\varphi}.$
Consequently,
\begin{eqnarray*}
{\mathfrak h}(\mu_{\varphi})  + \int \,\varphi \ d\mu_{\varphi} &=&  \left(\inf_{\psi \,\in \,\mathcal{B}_\Gamma}\, \int \psi \, d\mu_{\varphi}\right) + \frac{L(\varphi)}{L(1)} \,=\,  \left(\inf_{\psi \,\in \,\mathcal{B}_\Gamma} \, \frac{L(\psi)}{L(1)}\right) + \frac{L(\varphi)}{L(1)} \\
&& \medskip\\
&=&  \frac{1}{L(1)}\,\left(L(\varphi) + \inf_{\psi \,\in\, \mathcal{B}_\Gamma} \,L(\psi) \right) \geqslant  0 \,=\, \Gamma(\varphi)
\end{eqnarray*}
where we have used relation \eqref{eq:L} in the last step. This completes the proof of \eqref{eq:leq}.

\smallskip

Endowing $\mathcal{P}_a(X)$ with the total variation distance, the function ${\mathfrak h}$ is upper semi-continuous since it is defined as the infimum of the family of continuous functions $\left(\mu \in \mathcal{P}_a(X)\,  \mapsto  \,\int \varphi \, d\mu\right)_{\varphi \, \in \,\mathcal{A}_\Gamma}$ (cf. \cite{Ba,Bo}).

We proceed by showing the maximality (hence uniqueness) of the function ${\mathfrak h}$ among those which satisfy \eqref{eq:var}. Let $\alpha$ be such a map. Then,
$$\Gamma(-\psi) \,\,\geqslant\,\, \alpha(\mu) + \int -\psi \, d\mu \quad \quad \forall \,\psi \in {\mathbf B} \quad \forall \,\mu \in \mathcal{P}_a(X)$$
or, equivalently,
$$\alpha(\mu) \,\,\leqslant\,\, \Gamma(-\psi) + \int \psi \, d\mu \quad \quad \forall \,\psi \in {\mathbf B} \quad \forall \,\mu \in \mathcal{P}_a(X)$$
which implies that
\begin{equation}\label{eq:gamma-leq}
\alpha(\mu) \,\,\leqslant\,\, \inf_{\psi \,\in\, {\mathbf B}}\, \left\lbrace \Gamma(-\psi) + \int \psi \, d\mu \right\rbrace \quad \quad \forall\,\mu \in \mathcal{P}_a(X).
\end{equation}
Moreover, as $\mathcal{A}_\Gamma \subset {\mathbf B}$ and $\Gamma(-\psi) \,\,\leqslant\,\, 0$ for every $\psi \in \mathcal{A}_\Gamma$ we conclude that
\begin{eqnarray*}
\alpha(\mu) \,\,&\leqslant&\,\, \inf_{\psi \,\in\, {\mathbf B}}\, \left\lbrace \Gamma(-\psi) + \int \psi \, d\mu \right\rbrace \,\,\leqslant\,\, \inf_{\psi \,\in\, \mathcal{A}_\Gamma}\, \left\lbrace \Gamma(-\psi)+ \int \psi \, d\mu \right\rbrace \\
&\leqslant& \inf_{\psi \,\in\, \mathcal{A}_\Gamma} \,\left\lbrace \int \psi \, d\mu \right\rbrace \,=\, {\mathfrak h}(\mu) \quad \quad \forall\,\mu \in \mathcal{P}_a(X).
\end{eqnarray*}
The previous reasoning using $\alpha = \mathfrak{h}$ allows us to conclude that
\begin{equation}\label{eq:entropyB}
{\mathfrak h}(\mu) \,\,=\,\, \inf_{\psi \,\in\, {\mathbf B}}\, \left\lbrace \Gamma(-\psi) + \int \psi \, d\mu \right\rbrace \quad \quad \forall \,\mu \in \mathcal{P}_a(X).
\end{equation}

As ${\mathbf B}$ is a vector space the equality \eqref{eq:entropyB} can be rewritten as
$${\mathfrak h}(\mu) \,=\,  \inf_{\psi \,\in\, {\mathbf B}}\, \left\lbrace \Gamma(-\psi) - \int - \psi \, d\mu \right\rbrace
= \,  \inf_{\psi \,\in\, {\mathbf B}}\, \left\lbrace \Gamma(\psi) - \int  \psi \, d\mu \right\rbrace.$$

We are left to prove that $\mathfrak{h}$ is affine. The convex set $\cP_a(X)$ is compact in the weak$^*$ topology (cf. \cite[Theorem 2, V.4.2]{DS}) and so, by the Krein-Milman Theorem,
it is the closed convex hull of its extreme points. Moreover, we can use the the Choquet Representation Theorem (cf. \cite[Theorem 6.6]{AE80} or \cite[p.153]{Wa}) to express each member of $\cP_a(X)$ in terms of the extreme elements of $\cP_a(X)$. More precisely, if $E_a(X)$ denotes the set of extreme points of $\cP_a(X)$ and $\mu$ belongs to $\cP_a(X)$ then there is a unique measure $\mathbb{P}_\mu$ on the Borel subsets of the compact metrizable space $\cP_a(X)$ such that $\mathbb{P}_\mu(E_a(X))=1$ and
\begin{equation}\label{Choquet1}
\int_X \,\psi(x) \,d\mu(x) = \int_{E_a(X)}\,\left(\int_X\, \psi(x) \, dm(x)\right)\,d\mathbb{P}_\mu(m) \quad \quad \forall\, \psi \in B_m(X).
\end{equation}
Hence every $\mu \in \cP_a(X)$ is a generalized convex combination of extreme finitely additive probability measures. We write $\mu = \int_{E_a(X)}\,m\, \,d\mathbb{P}_\mu(m)$ and call this equality the decomposition in extremes of the finitely additive probability measure $\mu$.

\begin{lemma}\label{le:Wa(ii)} Given $\mu \in \cP_a(X)$ whose decomposition in extremes is $\mu = \int_{E_a(X)}\,m\, \,d\mathbb{P}_\mu(m)$, then
$$\mathfrak{h}(\mu) = \int \mathfrak{h}(m)\,d\mathbb{P}_\mu(m).$$
In particular, the function $\mathfrak{h}$ is affine.
\end{lemma}

\begin{proof}
Recall that $\mathfrak{h}(\mu) = \inf_{\psi \, \in \, \mathcal{A}_\Gamma} \,\int  \psi \, d\mu$ and the map $\psi \, \in \, B_m(X) \, \mapsto \, \int  \psi \, d\mu$ is continuous. By \eqref{Choquet1} there exists a probability measure $\mathbb P_\mu$ giving full weight to the space $E_a(X)$ of extreme measures of $\cP_a(X)$ and satisfying $\int \,\psi\, d\mu = \int_{E_a(X)}\,\left(\int\, \psi \, dm \right)\,d\mathbb{P}_\mu(m)$. In particular,
$$
\mathfrak{h}(\mu) = \inf_{\psi \, \in \, \mathcal{A}_\Gamma} \,\int_{E_a(X)}\,\left(\int \, \psi(x) \, dm(x)\right)\,d\mathbb{P}_\mu(m).
$$
Moreover, as $\mathbb P_\mu$ is countably additive, the Monotone Convergence Theorem (cf.~\cite[Theorem IV.15, Vol. I]{RS}) for nets of continuous maps when applied to the net $(\int  \psi \, d\mu)_{\psi \, \in \, \mathcal{A}_\Gamma}$ implies that
\begin{eqnarray*}
\mathfrak{h}(\mu) &=& \inf_{\psi \, \in \, \mathcal{A}_\Gamma} \,\int_{E_a(X)}\,\left(\int \, \psi(x) \, dm(x)\right)\,d\mathbb{P_\mu}(m) \\
	&=& \int_{E_a(X)}\,\left(\inf_{\psi \, \in \, \mathcal{A}_\Gamma} \,\int \, \psi(x) \, d\,m(x)\right)\,d\mathbb{P}_\mu(m)
= \int_{E_a(X)}\,\mathfrak{h}(m)\,d\mathbb{P}_\mu(m).
\end{eqnarray*}
Finally, by Choquet Theorem, given $\mu_1,\mu_2 \in \cP_a(X)$ there exist unique probability measures $\mathbb P_{\mu_1}$, $\mathbb P_{\mu_2}$ giving full weight to $E_a(X)$ and such that $\mu_i = \int_{E_a(X)}\,m\,\,d\mathbb{P}_{\mu_i}(m)$ for $i=1,2$. In particular, for each $0 < \alpha < 1$, one has
$$\alpha \mu_1 + (1-\alpha) \mu_2 \,=\, \int_{E_a(X)}\,m \, \, d\,[\alpha \,d\mathbb{P}_{\mu_1} + (1-\alpha) \,d\mathbb{P}_{\mu_2}](m)$$
and, using the first part of the lemma, one gets
$$\mathfrak{h}(\alpha \mu_1 + (1-\alpha) \mu_2)  = \int_{E_a(X)}\mathfrak{h}(m)d[\alpha \mathbb{P}_{\mu_1} +
(1-\alpha)\mathbb{P}_{\mu_2}](m) = \alpha \mathfrak{h}(\mu_1) + (1-\alpha) \mathfrak{h}(\mu_2).$$
This ends the proof of Theorem~\ref{thm:main}.
\end{proof}

\section{Tangent functionals: Proof of Theorem~\ref{thm:tangent-functionals}}\label{sec:tangent}

The argument we will present follows closely the one in \cite[Theorems 9.14 and 9.15]{Wa}. Consider $\varphi \in \mathbf B$ and assume that $\mu \in \cE_\varphi(\Gamma)$. Then, by Theorem~\ref{thm:main},
$$\Gamma(\varphi + \psi)-\Gamma(\varphi) \geqslant {\mathfrak h}(\mu) + \int (\varphi + \psi) \, d\mu -  {\mathfrak h}(\mu) - \int \varphi \, d\mu = \int \psi\, d\mu \quad \quad \forall \,\psi \in \mathbf B.$$
This shows that $\cE_\varphi(\Gamma)\subseteq \mathcal T_\varphi(\Gamma).$ To establish the converse inclusion, fix $\mu \in \mathcal T_\varphi(\Gamma)$ and note that
\begin{eqnarray*}
\Gamma(\varphi + \psi) - \Gamma(\varphi) &\geqslant& \int \psi\, d\mu \quad \quad \forall \,\psi \in \mathbf B \\
&\Updownarrow& \\
\Gamma(\varphi + \psi) - \int (\varphi +\psi) \, d\mu &\geqslant& \Gamma(\varphi) - \int \varphi\, d\mu \qquad \forall \psi \in \mathbf B.
\end{eqnarray*}
This equivalence together with the variational principle in Theorem~\ref{thm:main}, the fact that $\mathbf B$ is a vector space and equation \eqref{eq:entropyB} imply that
$${\mathfrak h}(\mu)=\inf_{\psi\,\in\, \mathbf B} \,\Big\{\Gamma(\varphi + \psi)- \int (\varphi +\psi) \, d\mu\Big\} \geqslant \Gamma(\varphi) - \int \varphi\, d\mu.$$
Since the reverse inequality
$${\mathfrak h}(\mu) \leqslant \Gamma(\varphi) - \int \varphi\, d\mu$$
is an immediate consequence of \eqref{eq:var}, we conclude that $\mathcal T_\varphi(\Gamma) \subseteq \cE_\varphi(\Gamma)$. The second claim in the statement of Theorem~\ref{thm:tangent-functionals} is a consequence of \cite{Mazur} (see also \cite[page 12]{Phe93}), which ensures that the convex function $\Gamma$, acting on the separable Banach space $\mathbf B=C_b(X)$ or $\mathbf B=C_c(X)$, admits a unique tangent functional for every $\varphi$ in a residual subset of $\mathbf B$.

\section{Fr\'echet differentiability: Proof of Theorem~\ref{thm:differentiable-functionals}}\label{sec:diff-func}

In this section we will show the characterization of Fr\'echet differentiability of the pressure functional in terms of tangent functionals. \\

\noindent $(a) \Rightarrow (b)$. As $\Gamma$ is locally affine at $\varphi$, there exist (a unique) $\mu_\varphi \in \cP_a(X)$ (respectively, a signed probability measure if $X$ is locally compact and $\mathbf B=C_c(X)$) such that, for every $\psi_1,\,\psi_2\in \mathbf B$ whose norms are small enough, one has
$$\Gamma(\varphi+\psi_1)-\Gamma(\varphi+\psi_2)= \int (\psi_1-\psi_2) \, \mu_\varphi.$$
This implies that $\mathcal T_{\varphi+\psi}(\Gamma)=\{\mu_\varphi\}$ for every $\psi \in \mathbf B$ with small enough norm. Thus,
$$\lim_{\psi \,\to \,0}\,\sup\,\Big\{\|\mu - \mu_\varphi\| \,\colon \,\,\mu \in \mathcal T_{\varphi+\psi}(\Gamma)\Big\} = \lim_{\psi \,\to \,0}\,\{0\} = 0.$$

\noindent $(b) \Rightarrow (c)$. Assume that there exists a unique tangent functional $\mu_\varphi\in \mathcal T_\varphi(\Gamma)$ and that
\begin{equation}\label{eq:hyp}
\lim_{\psi \,\to \,0}\,\sup\,\Big\{\|\mu - \mu_\varphi\| \,\colon \,\,\mu \in \mathcal T_{\varphi+\psi}(\Gamma)\Big\}= 0.
\end{equation}
As $\mu_\varphi \in \mathcal T_\varphi(\Gamma)$, one has
$$\Gamma(\varphi+\psi) -\Gamma(\varphi) \, \geqslant \, \int \psi\, d\mu_\varphi \quad \quad \forall\, \psi \in \mathbf B.$$
Moreover, the uniqueness of the tangent functional $\mu_\varphi$ at $\varphi$ and Theorem~\ref{thm:tangent-functionals} imply that
$$\Gamma(\varphi)=\mathfrak h(\mu_\varphi) + \int \varphi\; d\mu_\varphi	\quad \quad \text{and} \quad \quad	\Gamma(\varphi)>\mathfrak h(\mu) + \int \varphi\; d\mu \quad \,\forall\, \mu \neq \mu_\varphi.$$
So, given $\psi \in \mathbf B$ and $\mu \in \mathcal T_{\varphi+\psi}(\Gamma)=\mathcal E_{\varphi+\psi}(\Gamma)$, one gets
\begin{align*}
0 & \leqslant \Gamma(\varphi+\psi) - \Gamma(\varphi) - \int \psi\, d\mu_\varphi = \mathfrak h(\mu) + \int (\varphi +\psi) \; d\mu - \Gamma(\varphi) - \int \psi\, d\mu_\varphi  \\
	& \leqslant \mathfrak h(\mu) + \int (\varphi +\psi) \; d\mu - \mathfrak h(\mu) - \int \varphi\; d\mu -\int \psi\, d\mu_\varphi \\
	& = \int \psi\, d\mu - \int \psi\, d\mu_\varphi  \,\,\leqslant \,\,\|\psi\|_\infty \, \,\|\mu-\mu_\varphi\|.
\end{align*}
Therefore, by assumption \eqref{eq:hyp}, one has
$$\lim_{\psi\, \to\, 0} \, \frac{1}{\|\psi\|_\infty}\,\left|\Gamma(\varphi+\psi) -\Gamma(\varphi) -\int \psi\, d\mu_\varphi\right| = 0.$$

\noindent $(c) \Rightarrow (a)$. The map $\mathfrak{h}$ is affine, but the pressure function $\Gamma$ is not associated to an underlying dynamics. Therefore, part of the argument to prove Theorem 6 of \cite{Wa2} (the analogue of Theorem~\ref{thm:differentiable-functionals} for the topological pressure) has to be adapted to the general setting we are dealing with. Assume that $\Gamma$ is Fr\'echet differentiable at $\varphi \in \mathbf B$. Then, as $\Gamma$ is convex, there is a unique tangent functional to $\Gamma$ at $\varphi$ (cf. \cite[Chapt. IV, \S 44]{RV73}), say $\mathcal T_\varphi(\Gamma) = \{\mu_\varphi\}$, and $\mu_\varphi$ satisfies
$$\lim_{\psi\,\to\, 0} \,\frac1{\|\psi\|_\infty} \,\Big|\Gamma(\varphi+\psi) - \Gamma(\varphi) -\int \psi\, d\mu_\varphi\Big| = 0.$$
By \eqref{eq:tf1}, one also has
$$\Gamma(\varphi+\psi) -\Gamma(\varphi) \,\geqslant \, \int \psi \, d\mu_\varphi \qquad  \forall \,\psi \in \mathbf B.$$
We are left to prove the reverse inequality for $\psi$ inside a neighborhood of $0$ in $\mathbf B$.

\begin{lemma}\label{le:Wa(i)}
Let $(\mu_n)_{n \, \in \, \mathbb{N}}$ be a sequence in $\cP_a(X)$ such that $\lim_{n \, \to \, +\infty}\,\mathfrak{h}(\mu_n) + \int \varphi\, d\mu_n = \Gamma(\varphi)$. Then $\lim_{n \, \to \, +\infty}\,\|\mu_n - \mu_\varphi\|=0.$
\end{lemma}

\begin{proof} Given $\varepsilon > 0$, take $\delta > 0$ so that
$$\psi \in \mathbf B, \,\, \,\|\psi\|_\infty \,\leqslant\, \delta \quad \quad \Rightarrow \quad \quad 0 \,\leqslant\, \Gamma(\varphi + \psi) - \Gamma(\varphi) - \int \psi\, d\mu_\varphi \,\leqslant\, \varepsilon \,\|\psi\|_\infty.$$
For each $n \in \mathbb{N}$ consider $\varepsilon_n = \Gamma(\varphi) - \mathfrak{h}(\mu_n) - \int \varphi\, d\mu_n$, and let $N \in \mathbb{N}$ be such that for every $n \geqslant N$ one has $0 \leqslant \varepsilon_n < \varepsilon  \delta$. Therefore, if $\|\psi\|_\infty \leqslant \delta$ and $n \geqslant N$,
\begin{align*}
\int \psi\, d\mu_n - \int \psi\, d\mu_\varphi
    & = \Gamma(\varphi) + \int \psi\, d\mu_n - \Gamma(\varphi) - \int \psi\, d\mu_\varphi \\
	& = \mathfrak{h}(\mu_n) + \int \varphi\, d\mu_n + \varepsilon_n + \int \psi\, d\mu_n - \Gamma(\varphi) - \int \psi\, d\mu_\varphi\\
	& = \mathfrak{h}(\mu_n) + \int (\varphi + \psi)\, d\mu_n + \varepsilon_n - \Gamma(\varphi) - \int \psi\, d\mu_\varphi \\
	& \leqslant \Gamma(\varphi + \psi) + \varepsilon_n - \Gamma(\varphi) - \int \psi\, d\mu_\varphi \\
	& = \Gamma(\varphi + \psi) - \Gamma(\varphi) - \int \psi\, d\mu_\varphi + \varepsilon_n < 2 \delta \varepsilon.
\end{align*}
Since these estimates are also valid for $-\psi$, one has
$\Big|\int \psi\, d\mu_n - \int \psi\, d\mu_\varphi\Big| < 2  \delta \varepsilon$
for every $\|\psi\|_\infty \leqslant \delta$ and $n \geqslant N$.
Thus, if $n \geqslant N$,
\begin{eqnarray*}
\|\mu_n - \mu_\varphi\| &=& \sup \, \left\{\Big|\int \psi\, d\mu_n - \int \psi\, d\mu_\varphi\Big|\,\colon \, \|\psi\|_\infty \leqslant 1 \right\}\\
&=& \frac{1}{\delta}\,\sup \, \left\{\Big|\int \psi\, d\mu_n - \int \psi\, d\mu_\varphi\Big|\,\colon \, \|\psi\|_\infty \leqslant \delta \right\} \,< \, \frac{1}{\delta} \, 2  \delta \varepsilon  \,= \,2 \varepsilon.
\end{eqnarray*}
\end{proof}

Recall that, by Lemma~\ref{le:Wa(ii)}, one has $\mathfrak{h}(\mu) = \int \mathfrak{h}(m)\,d\mathbb{P}_\mu(m)$ for every $\mu \in \cP_a(X)$ whose decomposition in extreme points is $\mu = \int_{E_a(X)}\,m\, \,d\mathbb{P}_\mu(m)$. Thus,
\begin{equation}\label{eq:vp-extreme}
\Gamma(\varphi) = \sup \, \Big\{\mathfrak{h}(\mu) + \int \varphi \, d\mu \,\,\,| \,\,\, \mu \,\in \,E_a(X)\Big\}.
\end{equation}

\begin{lemma}\label{le:Wa(iii)}
The tangent functional $\mu_\varphi$ is an extreme point of $\cP_a(X)$ and
$$\Gamma(\varphi) > \sup \, \Big\{\mathfrak{h}(\mu) + \int \varphi \, d\mu \,\,\,| \,\,\, \mu \,\in \,E_a(X) \,\text{ and } \,\mu \neq \mu_\varphi\Big\}.$$
\end{lemma}

\begin{proof}
Using \eqref{eq:vp-extreme}, one can choose a sequence $(\mu_n)_{n \,\in \,\mathbb{N}}$ of extreme finitely additive probability measures with
$$
\lim_{n \,\to\, +\infty} \,\, \big(\mathfrak{h}(\mu_n) + \int \varphi \, d\mu_n\big) = \Gamma(\varphi).$$
Then, by Lemma~\ref{le:Wa(i)}, one has $\lim_{n \, \to \, +\infty}\,\|\mu_n - \mu_\varphi\|=0.$ Since distinct extreme points in $\cP_a(X)$ have norm distance equal to $2$ (cf. \cite{AE80}), there is $N \in \mathbb{N}$ such that $\mu_n = \mu_N$ for every $n \geqslant N$. Therefore, $\mu_N=\mu_\varphi$, so $\mu_\varphi$ is an extreme point of $\cP_a(X)$.
In addition, observe that, since $\mathcal T_\varphi(\Gamma)=\mathcal E_\varphi(\Gamma) = \{\mu_\varphi\}$, the previous argument also shows that one cannot have
$$\Gamma(\varphi) = \sup \, \Big\{\mathfrak{h}(\mu) + \int \varphi \, d\mu \,\,\,| \,\,\, \mu \,\in \,E_a(X) \,\text{ and } \,\mu \neq \mu_\varphi\Big\}.$$
Thus, $\Gamma(\varphi) > \sup \, \Big\{\mathfrak{h}(\mu) + \int \varphi \, d\mu \,\,\,| \,\,\, \mu \,\in \,E_a(X) \,\text{ and } \,\mu \neq \mu_\varphi\Big\}.$
\end{proof}

\begin{lemma}\label{le:Wa(iv)}
There is a neighborhood $\mathcal{U}$ of $\mu_\varphi$ in the total variation norm such that
$$\mathfrak{h}(\mu_\varphi) > \sup \, \Big\{\mathfrak{h}(\mu)  \,\,\,| \,\,\, \mu \in \mathcal{U}, \,\, \mu \,\in \,E_a(X) \text{ and } \,\mu \neq \mu_\varphi\Big\}.$$
\end{lemma}

\begin{proof}
Let $a=\Gamma(\varphi)- \sup \, \Big\{\mathfrak{h}(\mu) + \int \varphi \, d\mu \,\,\,| \,\,\, \mu \in E_a(X) \,\text{ and }\,\mu \neq \mu_\varphi\Big\}$ and $\mathcal{U}$ the neighborhood of $\mu_\varphi$ given by
$$\mathcal{U} = \Big\{\mu \in \cP_a(X) \, \colon \,\, \Big|\int \varphi \, d\mu - \int \varphi\, d\mu_\varphi\Big| < a/2\Big\}.$$
If $\mu \in \mathcal{U}$ is an extreme of $\cP_a(X)$ and $\mu \neq \mu_\varphi$, then
\begin{eqnarray*}
\mathfrak{h}(\mu) &\leqslant& \mathfrak{h}(\mu) + \int \varphi \, d\mu - \int \varphi \, d\mu_\varphi + \frac{a}{2} \,\,\leqslant \,\,\Gamma(\varphi) - a - \int \varphi \, d\mu_\varphi + \frac{a}{2} \\
&=& \Gamma(\varphi) - \int \varphi \, d\mu_\varphi - \frac{a}{2} = \mathfrak{h}(\mu_\varphi)- \frac{a}{2}
\end{eqnarray*}
where the last equality is due to the fact that $\mathcal T_\varphi(\Gamma)=\mathcal E_\varphi(\Gamma) = \{\mu_\varphi\}$.
\end{proof}

We are finally ready to show that $\Gamma$ is locally affine at $\varphi$. Let $a$ be as in the proof of Lemma~\ref{le:Wa(iv)}. Then, for every $\psi \in B_m(X)$ satisfying $\|\psi - \varphi\|_\infty < a/2$, one has
\begin{align*}
\sup\, \Big\{\mathfrak{h}(\mu) & +\int \psi\, d\mu \,\,\,| \,\,\, \mu \,\in \,E_a(X) \text{ and } \,\mu \neq \mu_\varphi\Big\} \\
&\leqslant  \Gamma(\varphi) - a + \|\psi - \varphi\|_\infty \leqslant \Gamma(\psi) - a + 2\,\|\psi - \varphi\|_\infty \,<\, \Gamma(\psi).
\end{align*}
Thus, by \eqref{eq:vp-extreme}, all such maps $\psi$ which are $a/2$-close to $\varphi$ have $\mu_\varphi$ as unique extreme $\Gamma$-equilibrium state in $\cP_a(X)$. In particular,
$$\|\psi - \varphi\|_\infty < a/2 \quad \quad \Rightarrow \quad \quad \Gamma(\psi) = \mathfrak{h}(\mu_\varphi) +\int \psi\, d\mu_\varphi.$$
So $\Gamma$ is locally affine at $\varphi$. This ends the proof of the first part of Theorem~\ref{thm:differentiable-functionals}.

Regarding the second list of equivalent assertions stated in Theorem~\ref{thm:differentiable-functionals}, firstly assume that $\Gamma$ is everywhere Fr\'echet differentiable. The previous equivalent conditions imply
that $\Gamma$ is locally affine at every $\varphi \in \mathbf B$, so, by the connectedness of the vector space $\mathbf B$, we conclude that $\Gamma$ is affine. Conversely, affine functions are clearly Fr\'echet differentiable. Hence items $(\bar{a})$ and $(\bar{b})$ are equivalent. Assume now that $\Gamma$ is affine. Then there exists $\mu \in \cP_a(X)$ such that for every $\varphi \in B$ one has
$$\Gamma(\varphi + \psi)-\Gamma(\varphi) = \int \psi\; d\mu \quad \quad \forall \, \psi \in B.$$
Thus $\mu$ is a tangent functional to every $\varphi \in \mathbf B$. As any element in $\cP_a(X)$ is determined by its integrals over $\mathbf B$, the previous equality implies that $\mu$ is the unique tangent functional at every $\varphi \in B$. So, $\bigcup_{\varphi\,\in\,\mathbf B} \,\mathcal T_\varphi(\Gamma) =\{\mu\}$, and $(\bar{a})$ implies $(\bar{c})$. Finally, condition $(\bar{c})$ implies $(\bar{a})$ due to the corresponding local property $(c) \Rightarrow (a)$.
This completes the proof of Theorem~\ref{thm:differentiable-functionals}.

\section{Pressure, entropy and zero temperature limits}\label{sse:top-pressure}

We start by defining the topological and free energy for a continuous transformation of a compact metric space, along with a recollection of some of their properties.

\subsection{{{Entropy and pressure}}}\label{sse:topol-pressure}

Let $f \,\colon X \to X$ be a continuous transformation of a compact metric space $(X,d)$. Given $n \in \mathbb{N}$, define the dynamical distance $d_n \colon X \times X \, \to \,[0,+\infty)$ by
$$d_n(x,z)=\max\,\Big\{d(x,z),\,d(f(x),f(z)),\,\dots,\,d(f^{n}(x),f^{n}(z))\Big\}$$
which generates the same topology as $d$. For every $x \in X$, $n \in \mathbb{N}$ and $\varepsilon > 0$, we denote by $B^f_n(x,\varepsilon)$ the open ball centered at $x$ with radius $\varepsilon$ with respect to the metric $d_n$.
Having fixed $n \in \mathbb{N}$ and $\varepsilon>0$, we say that a set $E \subset X$ is $(n,\varepsilon)$--separated by $f$ if
$d_n(x,z) \geqslant \varepsilon \quad \quad \forall \, x \neq z \,\in \,E.$
Denote by $s_n(f,\varepsilon)$ the maximal cardinality of all $(n,\varepsilon)$--separated subsets of $X$ by $f$. Due to the compactness of $X$, the number $s_n(f,\varepsilon)$ is finite.
The \emph{topological entropy} of $f$ is defined by
$$h_{\text{top}}(f) = \lim_{\varepsilon \, \to \, 0^+}\,\limsup_{n\, \to\, +\infty}\,\frac{1}{n}\,\log \,s_n(f,\varepsilon).$$

More generally, given a continuous map  $\varphi \,\colon X \to \mathbb{R}$  (also called a potential), the \emph{topological pressure} of $f$ and $\varphi$ is defined by
\begin{equation}\label{eqPf1}
P_{\text{top}}(f,\varphi) = \lim_{\varepsilon \, \to \, 0^+}\,\limsup_{n\, \to\, +\infty}\,\frac{1}{n}\,\log \,P_n(f,\varphi,\varepsilon)
\end{equation}
where, for every $n \in \mathbb{N}$,
\begin{equation}\label{eqPf2}
P_n(f,\varphi,\varepsilon)= \sup_{E}\,\Big\{\sum_{x \,\in\, E}\,e^{S_n^f \,\varphi(x)}\,\colon \text{ $E \subset X$ is $(n,\varepsilon)$-separated}\Big\}
\end{equation}
and
$S_n^f \,\varphi(x) = \varphi(x) + \varphi(f(x)) + \cdots + \varphi(f^n(x)).$
This way, one assigns to each point $x \in X$ the weight $e^{S^f_n\varphi(x)}$ determined by the potential $\varphi$ along the block of the first $n$ iterates of $f$ at $x$. In particular, $P_{\text{top}}(f,0) = h_{\text{top}}(f)$.

As $X$ is compact, $C(X)$ is a subspace of $B_m(X)$ where $\|.\|_\infty$ is the norm of the uniform convergence. The pressure map
$$P_{\text{top}}(f, .) \, \colon \,\,C(X) \quad \to \quad \mathbb{R}\cup\{+\infty\}$$
satisfies, for every $\varphi, \, \psi \in C(X)$ and constant $c \in \mathbb{R}$, the following properties \cite{Wa}:

\begin{enumerate}
\item  $h_{\text{top}}(f) + \min \,\varphi \,\,\leqslant \,\,P_{\text{top}}(f, \varphi) \,\,\leqslant \,\,h_{\text{top}}(f) + \max \,\varphi.$
\medskip
\item $P_{\text{top}}(f, .)$ is either finite valued or constantly $+\infty$.
\medskip
\item If $P_{\text{top}}(f, .) < +\infty$, then $P_{\text{top}}(f, .)$ is convex.
\medskip
\item $P_{\text{top}}(f, \varphi + c) = P_{\text{top}}(f, \varphi) + c$.
\medskip
\item $\varphi \leqslant \psi \quad \Rightarrow \quad P_{\text{top}}(f, \varphi)  \leqslant P_{\text{top}}(f, \psi).$
\medskip
\item $P_{\text{top}}(f, \varphi + \psi\circ f - \psi) = P_{\text{top}}(f, \varphi)$.
\medskip
\item If $P_{\text{top}}(f, .) < +\infty$, then $\left|P_{\text{top}}(f, \varphi) - P_{\text{top}}(f, \psi)\right| \leqslant \|\varphi - \psi\|_\infty.$
\end{enumerate}

\smallskip

We observe that \eqref{eqPf1} and \eqref{eqPf2} can be used to define a pressure function on the space $B_m(X)$ of bounded potentials which, by some abuse of notation, we still denote by $P_{\text{top}}(f, .) \, \colon \,\,B_m(X) \to  \mathbb{R}\cup\{+\infty\}$. To avoid any confusion, in order to distinguish the pressure functions we will mention their domains.

Denote by $\mathcal{P}_f(X)$ the space of $f$-invariant Borel probability measures on $X$ endowed with the weak$^*$ topology.
Given $\mu \in \mathcal{P}_f(X)$ and a continuous potential $\varphi: \, X \to \mathbb{R}$, the \emph{free energy} of $f$, $\mu$ and $\varphi$ is given by
$$P_\mu(f,\varphi) = h_{\mu}(f) + \int \varphi \,d\mu$$
where $h_{\mu}(f)$ is the \emph{metric entropy} of $f$ with respect to $\mu$ (definition and properties may be read in \cite[Chapter 4]{Wa}).
A measure $\mu \in \mathcal{P}_f(X)$ is called an \emph{equilibrium state} for 
the potential $\varphi$ if
$P_{\mu}(f,\varphi) = \sup_{\nu \, \in \, \mathcal{P}_f(X)} \,\,\left\{h_{\nu}(f) + \int \varphi \,d\nu\right\}.$

\subsubsection{\textbf{\emph{Star-entropy}}}\label{sse:star}
The lack of upper semi-continuity of the entropy map has led some authors to regularize the notion of metric entropy. For instance, given $\mu \in \mathcal{P}_f(X)$, the concept of \emph{star-entropy} was introduced in \cite{New89} and later explored by Viana and Yang in \cite{VY}, being defined by
$$h^*_\mu(f)=\sup  \Big\{\limsup_{n  \,\to \, +\infty} \,h_{\mu_n}(f) \,\,|\,\,(\mu_n)_{n \,\in\, \mathbb{N}} \text{ is a sequence in $\mathcal{P}_f(X)$ with $\lim_{n  \,\to  \,+\infty}  \mu_n = \mu$}\Big\}.$$
It is known that, for every $\mu \in \mathcal{P}_f(X)$, one has
\begin{equation}\label{eq:Buzzi}
h_\mu(f) \,\,\leqslant \,\,h^*_\mu(f) \,\,\leqslant \,\,h_\mu(f) + h_{loc}(f)
\end{equation}
where $h_{loc}(f)$ stands for the local entropy of $f$, defined by
$$h_{loc}(f) = \lim_{\varepsilon \, \to \, 0^+} \,\lim_{\delta \, \to \, 0^+} \, \limsup_{n \, \to \, +\infty} \,\frac{1}{n}\,\sup_{x \, \in \, X}\, \log \, s_n(f,\delta,B^f_n(x,\varepsilon))$$
and $s_n(f,\delta,B^f_n(x,\varepsilon))$ denotes the maximal cardinality of an $(n, \delta)$-separated subset of $B^f_n(x,\varepsilon)$.  The first inequality of \eqref{eq:Buzzi} is a straightforward consequence of the definition of $h^*$, while the second was proved in \cite{New89} (see also \cite[Appendix B]{Buz}), from which we conclude that $h_{loc}(f)$ bounds the defect in upper semi-continuity of the map $\mu \in \mathcal{P}_f(X)\, \mapsto \, h_\mu(f)$. The star-entropy function $h^* :\cP_f(X) \to \mathbb R$ is related to the entropy structures of Boyle and Downarowicz (cf.~\cite{BD}) since $h^*$ is precisely the upper semi-continuous envelope of the metric entropy. More precisely, when the topological entropy of $f$ is finite the map $h^*$ is bounded from above by $T = \text{ the topological entropy of $f$}$ and one has (see \cite[p. 466--467]{DN})
\begin{equation}\label{eq:envelope}
h^* = \inf\,\Big\{T: \cP_f(X) \to \mathbb R \,\,|\,\, \text{$T$ is continuous and} \, \, T(\mu)\geqslant h_\mu(f) \;\; \forall \mu\in \cP_f(X)\Big\}.
\end{equation}
The advantage of considering the star-entropy is that the function $\mu \in \mathcal{P}_f(X) \mapsto h^*_\mu(f)$ is upper semi-continuous when $\mathcal{P}_f(X)$ is endowed with the weak$^*$-topology. In particular, defining the star-pressure by
$$\Ptop^*(f,\varphi) =\sup_{\mu \,\in \,\mathcal{P}_f(X)} \,\left\{ h^*_\mu(f) + \int \varphi \, d\mu \right\}$$
one guarantees that there always exists an $f$-invariant probability measure which attains the supremum. Yet, the behavior of the star-entropy function may differ substantially from the usual metric entropy map (we refer the reader to \cite{Rt2002} for examples of smooth interval maps without an invariant probability measure with maximal entropy), and so the corresponding maximum values and maximal measures might fail to describe standard physical quantities.

Denote the topological entropy of $f$ by $h_{\mathrm{top}}(f)$ and assume that it is finite. Let ${h}^{**}$ be the upper semi-continuous concave envelope of the metric entropy, defined in $\mathcal{P}_f(X)$ by
\begin{equation}\label{eq:concave-envelope}
h^{**} = \inf\Big\{T: \cP_f(X) \to \mathbb R \,\,|\,\, \text{$T$ is continuous, affine and} \, \, T(\mu)\geqslant h_\mu(f) \;\; \forall \mu\in \cP_f(X)\Big\}.
\end{equation}
Thus
$$h_\mu(f) \,\,\leqslant \,\, {h}^*_\mu(f) \,\,\leqslant \,\,h^{**}_\mu(f) \,\,\leqslant \,\, h_{\mathrm{top}}(f) \quad \forall \mu\in \cP_f(X).$$

\subsection{Ergodic optimization}\label{sec:erg-opt}

Given $\varphi \in C(X)$, the map $\mu  \in  \cP_f(X) \, \mapsto \, \int \varphi\, d\mu$ is continuous and defined on a compact metric space. Hence it has a maximum, which is realized by $f$-invariant probability measures. These are referred to as \emph{$\varphi$-maximizing probability measures}, and may be not unique. It is therefore useful to find $\varphi$-maximizing probability measures which are the most chaotic, meaning those which carry larger metric entropy. It is known that, if the metric entropy map is upper semi-continuous, then the $\varphi$-maximizing probability measures obtained through zero-temperature limits have this property (cf. \cite[Theorem~4.1]{Je}). More precisely, on the one hand, if
$\htop(f)< +\infty$, then
\begin{align*}\frac1t \Ptop(f,t\varphi) & = \sup_{\nu\,\in\, \cP_f(X)}\, \Big\{\frac1t \,h_{\nu}(f) + \int \varphi \, d\nu \Big\} \\
&
\underset{t\to +\infty}{\longrightarrow} \quad \sup_{\nu\,\in\, \cP_f(X)} \,\int \varphi\, d\nu = \max_{\nu\,\in\, \cP_f(X)} \,\int \varphi\, d\nu.
\end{align*}
On the other hand, if for each $t>0$ large enough there exists a unique equilibrium state $\mu_t$ for $f$ with respect to $t\varphi$, then any weak$^*$ accumulation probability measure $\mu\in\cP_f(X)$ of $(\mu_t)_{t\,>\,0}$ as $t$ goes to $+\infty$ is a $\varphi$-maximizing probability measure.

Two difficulties to carry on this approach for dynamical systems which may not be expansive is that the entropy function may
not be upper semi-continuous, and so equilibrium states likely fail to exist at sufficiently small temperatures. We refer the reader to \cite{VV10} for a class of non-uniformly expanding maps for which the entropy function is not upper semi-continuous, and for which one can only ensure that $t\varphi$ has an equilibrium state for small values of the parameter $t$. The results in Theorem~\ref{thm:second-main} guarantee a way to bypass these issues using the variational
$\mathfrak h$-entropy, which still satisfies the variational principle with respect to the classical pressure function.

\section{Variational principles: Proof of Theorem~\ref{thm:second-main} }\label{se:thermo-formalism}

In the mid seventies the thermodynamic formalism was brought from statistical mechanics to dynamical systems by the pioneering work of Sinai, Ruelle and Bowen \cite{Bo75}, which established a powerful correspondence between one-dimensional lattices and uniformly hyperbolic dynamics and conveyed several notions from one setting to the other. The success of this approach ultimately relies on a variational principle for the topological pressure, along with the construction of equilibrium states as the class of pressure maximizing invariant probability measures.
In this section we first show that Theorem~\ref{thm:main} extends the classical thermodynamic formalism for continuous maps on compact metric spaces, and then we complete the proof of Theorem~\ref{thm:second-main}.

\subsection{Classical variational principle}\label{sec:classic-1}

Given a continuous transformation $f: \,X \to X$ acting on a compact metric space $(X,d)$, the variational principle (cf. \cite[\S 9.3]{Wa}) states that, given a continuous potential $\varphi: \, X \to \mathbb{R}$,
$$\Ptop(f,\varphi)=\sup_{\mu\,\in\,\mathcal{P}_f(X)} \,\Big\{h_\mu(f) + \int_X \,\varphi \, d\mu \Big\}.$$
Moreover, the previous least upper bound coincides with the supremum evaluated on the set of ergodic probability measures (cf. \cite[Corollary 9.10.1]{Wa}). A probability measure which attains the maximal value is called an equilibrium state for $f$ and the potential $\varphi$. For instance, an equilibrium state for $f$ and the potential $\varphi\equiv 0$ is a measure with maximal entropy.

In the event that $\htop(f) < + \infty$ and the entropy function $\mu \in \mathcal{P}_f(X)\, \mapsto\, h_\mu(f)$ is upper semi-continuous, the metric entropy satisfies
(see \cite[Theorem 9.12]{Wa})
\begin{align*}\label{eq:dual}
h_\mu(f) &= \inf_{\varphi \in C(X)} \Big\{\Ptop(f,\varphi) - \int \varphi  d\mu \Big\} \\
&= \sup  \big\{\limsup_{n \, \to \, +\infty}\,h_{\mu_n} \,\,|\,\,(\mu_n)_{n \,\in\, \mathbb{N}} \text{ is a sequence in $\mathcal{P}_f(X)$ with $\,\,\lim_{n \, \to \, +\infty}  \mu_n = \mu$}\big\}.
\end{align*}
In this case, an easy computation using \cite[Theorem~9.12]{Wa} also yields
\begin{equation*}\label{eq:VE}
h_\mu(f) = \inf_{\varphi \,\in \,\Lambda} \,\int \varphi \, d\mu
\end{equation*}
where
$\Lambda=\{\Ptop(f,\varphi) - \varphi \colon \,\varphi \in C(X)\}.$

\subsection{New variational principle}\label{sse:theorem 5}
Keeping the classical notion of topological pressure and summoning Theorem~\ref{thm:main}, we replace the entropy map (metric or star) acting on the space $\mathcal{P}_f(X)$ by a more general real valued function ${\mathfrak h}_f$ whose domain is the space $\mathcal{P}(X)$ of the Borel probability measures on $X$. More precisely, assume that $\htop(f) < +\infty$; then $\Ptop(f, .): \,C(X) \to \mathbb{R}$ is a pressure function (cf.  Section~\ref{sse:top-pressure})
to which we may apply Theorem~\ref{thm:main}. This way, we conclude that the map ${\mathfrak h}_f :\, \mathcal{P}(X) \to \mathbb{R}$ given by
$${\mathfrak h}_\mu(f) = \inf_{\varphi \, \in \, \mathcal{A}_{\Ptop}}\, \int \varphi \, d\mu$$
where $\mathcal{A}_{\Ptop} = \left\{\varphi \in C(X) :\, \Ptop(f,-\varphi) \leqslant 0 \right\}$
is upper semi-continuous, satisfies
\begin{equation}\label{eq:second-dual-2}
{\mathfrak h}_\mu(f) = \inf_{\varphi \, \in \, C(X)}\, \left\{\Ptop(f, \varphi) - \int \varphi \, d\mu\right\}\quad \quad \forall \, \mu \in \mathcal{P}(X)
\end{equation}
and
\begin{equation}\label{eq:Pgamma}
\Ptop(f,\varphi) = \max_{\mu \,\in\, \mathcal{P}(X)}
	\,\Big\{{\mathfrak h}_\mu(f) + \int \varphi \, d\mu \Big\} \quad \quad \forall \,\varphi \in C(X).
\end{equation}

It is immediate from \eqref{eq:second-dual-2} that ${\mathfrak h}_\mu(f)\leqslant \htop(f)$ for every $\mu \in \mathcal{P}(X)$. Moreover, using the aforementioned strategy, it is clear that, given $\varphi \in C(X)$, there exists $\mu_\varphi \in \mathcal{P}(X)$ such that $\Ptop(f,\varphi) = {\mathfrak h}_{\mu_\varphi}(f) + \int \varphi \, d\mu_\varphi.$
In the special case of $\varphi \equiv 0$ one gets both the equality
\begin{equation*}\label{eq:Hgamma}
\htop(f) = \max_{\mu\, \in\, \mathcal{P}(X)} \,{\mathfrak h}_\mu(f)
\end{equation*}
and $\mu_0 \in \mathcal{P}(X)$ where ${\mathfrak h}_f$ attains its maximum value $\htop(f)$. This ends the proof of Theorem~\ref{thm:second-main}.

\begin{corollary}\label{cor:2dyn}
Given $\varphi \in C(X)$, every $\mu_\varphi \in \mathcal{P}(X)$ attaining the maximum at \eqref{eq:Pgamma} is $f$-invariant.
\end{corollary}

\begin{proof}
Recall that $\mu \in \mathcal{P}(X)$ is said to be $f$-invariant if $\int (\psi \circ f) \, d\mu = \int \psi \, d\mu$ for every $\psi$ in $C(X)$. Fix $\mu_\varphi \in \mathcal{P}(X)$ such that $\Ptop(f,\varphi) = {\mathfrak h}_{\mu_\varphi}(f) + \int \varphi \, d\mu_\varphi$, and consider $\psi \in C(X)$. By the variational relation \eqref{eq:Pgamma} applied to both $\varphi + \psi \circ f - \psi$ and $\varphi + \psi - \psi \circ f$ we may take $\mu_1, \, \mu_2 \in \mathcal{P}(X)$ such that
$$\Ptop(f,\varphi + \psi \circ f - \psi) = {\mathfrak h}_{\mu_1}(f) + \int \varphi\, d\mu_1 + \int (\psi\circ f) \, d\mu_1 - \int \psi \, d\mu_1$$
and
$$\Ptop(f,\varphi + \psi - \psi \circ f) = {\mathfrak h}_{\mu_2}(f) + \int \varphi \, d\mu_2 + \int \psi \, d\mu_2 - \int (\psi\circ f) \, d\mu_2.$$
Using the equalities
$$\Ptop(f, \varphi + \psi \circ f - \psi) = \Ptop(f, \varphi) = \Ptop(f, \varphi + \psi - \psi \circ f)$$
(cf. Subsection~\ref{sse:top-pressure}) together with \eqref{eq:Pgamma}, we conclude that
\begin{eqnarray*}
{\mathfrak h}_{\mu_\varphi}(f) + \int \varphi \, d\mu_\varphi &=& {\mathfrak h}_{\mu_1}(f) + \int \varphi \, d\mu_1 + \int (\psi\circ f) \, d\mu_1 - \int \psi \, d\mu_1  \\
&\geqslant&  {\mathfrak h}_{\mu_\varphi}(f) + \int \varphi \, d\mu_\varphi + \int (\psi \circ f) \, d\mu_\varphi - \int \psi \, d\mu_\varphi
\end{eqnarray*}
so $ \int (\psi \circ f) \, d\mu_\varphi - \int \psi \, d\mu_\varphi \leqslant 0.$ In a similar way, we deduce that
\begin{eqnarray*}
{\mathfrak h}_{\mu_\varphi}(f) + \int \varphi \, d\mu_\varphi &=& {\mathfrak h}_{\mu_2}(f) + \int \varphi \, d\mu_2 + \int \psi \, d\mu_2  - \int (\psi\circ f) \, d\mu_2 \\
&\geqslant&  {\mathfrak h}_{\mu_\varphi}(f) + \int \varphi \, d\mu_\varphi + \int \psi \, d\mu_\varphi - \int (\psi \circ f) \, d\mu_\varphi
\end{eqnarray*}
hence $\int \psi \, d\mu_\varphi - \int (\psi \circ f) \, d\mu_\varphi \leqslant 0.$
\end{proof}

A straightforward consequence one draws from the classical variational principle and the fact that the metric entropy is always non-negative is that the pressure function determines $\mathcal{P}_f(X)$, in the sense that, if $h_{\mathrm{top}}(f) < +\infty$, then (cf. \cite[Theorem 9.11]{Wa})
$$\mu \, \in \, \mathcal{P}_f(X) \quad \Leftrightarrow \quad \int \varphi \, d\mu \,\,\leqslant\,\, \Ptop(f, \varphi) \quad \forall \,\varphi \in C(X).$$
Therefore, by \eqref{eq:second-dual-2}, if $h_{\mathrm{top}}(f) < +\infty$ and $\mu$ belongs to $\mathcal{P}(X)$ then
$$\mu \in \mathcal{P}_f(X) \quad \Leftrightarrow \quad {\mathfrak h}_\mu(f) \geqslant 0.$$

\subsection{Linking ${\mathfrak h}$ and $h^{**}$}

Here we prove the second part of Theorem~\ref{thm:second-main}, relating the $h^{**}$-entropy (defined by \eqref{eq:concave-envelope}) with the ${\mathfrak h}$-entropy function as a consequence of the following result.

\begin{proposition}\label{prop:max-gamma}
Let $f: \,X \to X$ be a continuous transformation acting on a compact metric space $X$ with $\htop(f) < +\infty$. Then:
\begin{enumerate}
\item[\emph{(a)}] $\Ptop(f, \varphi) = \max_{\mu \,\in \,\mathcal{P}_f(X)}\,\,\Big\{h^*_f(\mu) + \int \varphi \, d\mu\Big\} \quad \forall\, \varphi \in C(X).$
\medskip
\item[\emph{(b)}] $h_\mu(f) \,\,\leqslant \,\, h^*_\mu(f) \,\,\leqslant \,\, h^{**}_\mu(f) \,\,=\,\, {\mathfrak h}_\mu(f) \quad \forall \, \mu \in \mathcal{P}_f(X).$
\end{enumerate}
\end{proposition}

The remainder of this subsection is devoted to the proof of this proposition. We start relating ${\mathfrak h}_\mu(f)$, $h_\mu(f)$ and $h^*_\mu(f)$ when $\mu$ belongs to $\mathcal{P}_f(X)$. As previously mentioned, for every $\mu \in \mathcal{P}_f(X)$ one has $h_\mu(f) \leqslant h^*_\mu(f)$. Moreover, by the classical variational principle, for each $\mu \in \mathcal{P}_f(X)$ we get
$$h_\mu(f) \,\,\leqslant \,\,\Ptop(f, \varphi) - \int \varphi \, d\mu  \quad \quad \forall \,\varphi \, \in \, C(X).$$
So
\begin{equation}\label{ineqhh-2}
{\mathfrak h}_\mu(f) = \inf_{\varphi \, \in \, C(X)}\, \left\{\Ptop(f, \varphi) - \int \varphi \, d\mu\right\} \,\,\geqslant \,\, h_\mu(f).
\end{equation}
If $h^*_\mu(f)  >  {\mathfrak h}_\mu(f)$ for some $\mu \in \mathcal{P}_f(X)$, then there exists $\nu \in \mathcal{P}_f(X)$ satisfying $h_\nu(f) > {\mathfrak h}_\nu(f)$, contradicting ~\eqref{ineqhh-2}. This proves that $h_\mu(f) \leqslant h^*_\mu(f) \leqslant  {\mathfrak h}_\mu(f)$ for every $\mu \in \mathcal{P}_f(X)$. These inequalities, together with Theorem~\ref{thm:main}, yield
\begin{align*}
\sup_{\mu \,\in \,\mathcal{P}_f(X)}\,\left\{h_\mu(f) + \int \varphi \, d\mu\right\}\,\,& \leqslant \,\,\max_{\mu \,\in \,\mathcal{P}_f(X)}\,\left\{h^*_\mu(f) + \int \varphi \, d\mu\right\} \\
& \leqslant \,\, \max_{\mu \,\in \,\mathcal{P}_f(X)}\,\left\{{\mathfrak h}_\mu(f) + \int \varphi \, d\mu\right\}.
\end{align*}
On the other hand, as $\mathcal{P}_f(X) \, \subset \, \mathcal{P}(X)$, one has for every $\varphi  \in  C(X)$
$$\max_{\mu \,\in \,\mathcal{P}_f(X)}\,\left\{{\mathfrak h}_\mu(f) + \int \varphi \, d\mu \right\} \,\,\leqslant \,\, \max_{\mu \,\in \,\mathcal{P}(X)}\,\left\{{\mathfrak h}_\mu(f) + \int \varphi \, d\mu\right\}.$$
Therefore, from both the classical and the new variational principle ~\eqref{eq:Pgamma} we deduce that
\begin{eqnarray*}
\max_{\mu \,\in \,\mathcal{P}_f(X)}\,\left\{{\mathfrak h}_\mu(f) + \int \varphi \, d\mu\right\} \, &\leqslant& \, \max_{\mu \,\in \,\mathcal{P}(X)}\,\left\{{\mathfrak h}_\mu(f) + \int \varphi \, d\mu\right\} \\
\,&=&\, \Ptop(f, \varphi) \,\,=\,\, \sup_{\mu \,\in \,\mathcal{P}_f(X)}\,\left\{h_\mu(f) + \int \varphi \, d\mu\right\}\\
&\leqslant&  \,\max_{\mu \,\in \,\mathcal{P}_f(X)}\,\left\{h^*_\mu(f) + \int \varphi \, d\mu\right\} \\
&\leqslant& \, \max_{\mu \,\in \,\mathcal{P}_f(X)}\,\left\{{\mathfrak h}_\mu(f) + \int \varphi \, d\mu\right\}.
\end{eqnarray*}
and, consequently,
\begin{equation}\label{eq:igual}
\Ptop(f, \varphi) \,=\, \max_{\mu \,\in \,\mathcal{P}_f(X)}\,\left\{{\mathfrak h}_\mu(f) + \int \varphi \, d\mu\right\} \,=\, \max_{\mu \,\in \,\mathcal{P}_f(X)}\,\left\{h^*_\mu(f) + \int \varphi \, d\mu\right\}.
\end{equation}
Notice that the supremum is attained since ${\mathfrak h}_f$ is upper semi-continuous on $\mathcal{P}_f(X)$, so the map $\mu \in \mathcal{P}_f(X) \,\mapsto\, {\mathfrak h}_\mu(f) + \int \varphi \, d\mu$ is upper semi-continuous as well; similar comment regarding $h^*$. This proves item (a) of the proposition.

\begin{remark}
The previous argument also works if we replace $h^*$ by $h^{**}$. Thus
\begin{equation}\label{eq:igual-2}
\Ptop(f, \varphi) \,=\, \max_{\mu \,\in \,\mathcal{P}_f(X)}\,\left\{h^{**}_\mu(f) + \int \varphi \, d\mu\right\}.
\end{equation}
\end{remark}

\medskip

We proceed by showing that $h^{**} \,= \,\mathfrak h$ in $\mathcal{P}_f(X)$. According to \cite[A.2.1]{D11}, the map $h^{**}$ is smaller than any other upper semi-continuous concave map that also upper bounds the metric entropy. Therefore, one has $h^{**} \,\leqslant \,\mathfrak h$, since $\mathfrak h$ is upper semi-continuous, affine (hence concave) and satisfies \eqref{ineqhh-2}. So, regarding item (b) of the proposition, we are left to show that $h^{**} \,\geqslant \,\mathfrak h$ in $\mathcal{P}_f(X)$.

To prove this inequality, we just need to adapt the argument in \cite[Theorem 9.12]{Wa} and this way conclude that, for every $\mu \,\in \,\mathcal{P}_f(X)$,
$$h^{**}_\mu(f) \,\,\geqslant\,\, \inf_{\varphi \, \in \, C(X)}\, \Big\{\Ptop(f, \varphi) - \int \varphi \, d\mu\Big\}$$
which, due to \eqref{eq:second-dual-2}, yields $h^{**}_\mu(f) \,\,\geqslant\,\, {\mathfrak h}_\mu(f).$

Take $\mu_0 \in \mathcal{P}_f(X)$ and, as $\htop(f) < +\infty$, consider a real number $b > h^{**}_{\mu_0}(f)$. Let $C$ be the set
$$C \,=\, \Big\{(\mu, t) \in \mathcal{P}_f(X) \times \mathbb{R}\colon \,\, 0 \,\leqslant\, t\, \leqslant \,h^{**}_{\mu}(f)\Big\}.$$
Since $h^{**}$ is concave, $C$ is a convex set: if $(\mu, t)$ and $(\nu, s)$ belong to $C$ and $0 \leqslant p \leqslant 1$, then $p\mu + (1-p)\nu \in \mathcal{P}_f(X)$ and
$$0\,\leqslant \, pt + (1-p)s \,\leqslant \, ph^{**}_\mu(f) + (1-p) h^{**}_\nu(f) \,\leqslant \, h^{**}_{p\mu + (1-p)\nu}(f).$$
Moreover, if we consider $C$ as a subset of the dual $C(X)^*$ with the weak$^*$ topology, then $(\mu_0, b)$ does not belong to the closure $\overline{C}$ of $C$. Indeed, if there was a sequence $\big((\mu_n, t_n)\big)_n \in C$ converging to $(\mu_0, b)$, then, since $h^{**}$ is upper semi-continuous, one would get
$$b\,=\, \lim_{n \,\to\,+\infty}\,t_n \,\leqslant \, \limsup_{n \,\to\,+\infty}\,h^{**}_{\mu_n}(f) \,\leqslant\, h^{**}_{\mu_0}(f)$$
contradicting the choice of $b$. Therefore, there is a continuous functional $F: \,C(X)^* \times \mathbb{R} \to \mathbb{R}$ separating the disjoint convex sets $\overline{C}$ and $\{(\mu_0, b)\}$, that is, such that
\begin{equation}\label{eq:F}
F((\mu, t)) \, < \, F((\mu_0, b)) \quad \quad \forall \,(\mu, t) \in \overline{C}.
\end{equation}
Since we are using the weak$^*$ topology on $C(X)^*$, there are $\psi \in C(X)$ and $\gamma \in \mathbb{R}$ such that $F((\mu, t)) \,=\, \int \psi \, d\mu + t\gamma$ for every  $(\mu, t) \in C(X)^* \times \mathbb{R}$. Thus \eqref{eq:F} becomes
$$\int \psi \, d\mu + t\gamma \,\,< \,\, \int \psi \, d\mu_0 + b\gamma \quad \quad \forall\, (\mu, t) \in \overline{C}.$$
In particular,
$$\int \psi \, d\mu + h^{**}_\mu(f)\gamma \,\,< \,\, \int \psi \, d\mu_0 + b\gamma \quad \quad \forall\, \mu \in \mathcal{P}_f(X)$$
so, when $\mu = \mu_0$, one obtains
$$\int \psi \, d\mu_0 + h^{**}_{\mu_0}(f)\gamma \,\,< \,\, \int \psi \, d\mu_0 + b\gamma.$$
Therefore $\gamma > 0$, and so
$$\int \frac{\psi}{\gamma} \, d\mu + h^{**}_{\mu}(f)  \,\,< \,\, \int \frac{\psi}{\gamma} \, d\mu_0 + b \quad \quad \forall\, \mu \in \mathcal{P}_f(X).$$
Consequently, by the variational principle \eqref{eq:igual-2},
$$\Ptop\big(f, \frac{\psi}{\gamma}\big) \,\leqslant \, \int \frac{\psi}{\gamma} \, d\mu_0 + b.$$
Hence,
$$b \, \geqslant \, \Ptop\big(f, \frac{\psi}{\gamma}\big) -  \int \frac{\psi}{\gamma} \, d\mu_0  \,\geqslant \, \inf_{\varphi \, \in \, C(X)}\, \Big\{\Ptop(f, \varphi) - \int \varphi \, d\mu_0\Big\}$$
which implies that
$$h^{**}_{\mu_0}(f) \, \geqslant \, \inf_{\varphi \, \in \, C(X)}\, \Big\{\Ptop(f, \varphi) - \int \varphi \, d\mu_0\Big\}.$$
The proof of the proposition is complete.

\begin{remark}
If $\mu_0 \in \mathcal{P}_f(X)$ is an equilibrium state for $f$ and $\varphi$ with respect to the classical variational principle, that is,
$$\Ptop(f,\varphi) = h_f(\mu_0) + \int \varphi \, d\mu_0$$
then $\mu_0$ attains the maximum at \eqref{eq:Pgamma} as well. Indeed,
$$\Ptop(f,\varphi) \geqslant {\mathfrak h}_{\mu_0}(f) + \int \varphi \, d\mu_0 \geqslant h_f(\mu_0) + \int \varphi \, d\mu_0 = \Ptop(f,\varphi).$$
The converse is false, though, as the example in Subsection~\ref{sec:non-standard} illustrates.
\end{remark}

\begin{remark}\label{rmk:otherbound}
When the metric entropy map is not upper semi-continuous, the map ${\mathfrak h}_f = h^{**}$ is a strict upper bound for the metric entropy. Thus, Theorem~\ref{thm:second-main} motivates the search for an optimal upper semi-continuous upper bound. Clearly, Theorem~\ref{thm:main} applied to the pressure function $P_{\text{top}}(f, .) \, \colon \,\,B_m(X) \to  \mathbb{R}\cup\{+\infty\}$ provides in general a better bound than ${\mathfrak h}_f$, since it guarantees that
\begin{equation*}
P_{\text{top}}(f, \varphi) = \max_{\mu \, \in \, \mathcal{P}_a(X)}\, \left\{ \mathfrak h_\infty(\mu) + \int \varphi \, d\mu \right\} \quad \quad \forall\, \varphi \in {B_m(X)}
\end{equation*}
where the upper semi-continuous map $\mathfrak h_{\infty}$ is defined by
$$\mathfrak h_{\infty}(\mu) = \inf_{\varphi \, \in \, {B_m(X)}}\, \left\{P_{\text{top}}(f, \varphi) - \int\varphi \, d\mu\right\} \quad \quad \forall \,\mu \in  \mathcal{P}_a(X)$$
and so $h_\mu(f) \leqslant \mathfrak h_{\infty}(\mu) \leqslant {\mathfrak h}_\mu(f)$ for every $\mu \in \cP(X).$
\end{remark}

\subsection{An example without a measure with maximal entropy}\label{sec:non-standard}

Given $\varphi \in C(X)$, denote by $\mathcal{P}_\varphi(f,X) \subset \mathcal{P}_f(X)$ the space of (classical) equilibrium states for $f$ and $\varphi$. Both $\mathcal{P}_\varphi(f,X)$ and $\mathcal{T}_\varphi(P_{\mathrm{top}})$ are convex sets, but whereas $\mathcal{T}_\varphi(P_{\mathrm{top}})$ is always non-empty and compact for the weak$^*$ topology, this is sometimes not true for $\mathcal{P}_\varphi(f,X)$, as we will now check. In general, one has $\mathcal{P}_\varphi(f,X) \subset \mathcal{T}_\varphi(P_{\mathrm{top}})$, with equality if and only if the metric entropy map is upper semi-continuous at every element of $\mathcal{T}_\varphi(P_{\mathrm{top}})$ (cf. \cite{Wa}). 
As stated in Theorem~\ref{thm:tangent-functionals}, it is the set $\mathcal{E}_\varphi(\Gamma)$
of the $f$-invariant probability measures which maximize the map $\Gamma (\varphi) = \max_{\mu \, \in \, \mathcal{P}_f(X)}\, \left\{ {\mathfrak h}_\mu(f) + \int \varphi \, d\mu \right\}$ that fills in the gap between $\mathcal{P}_\varphi(f,X)$ and $\mathcal{T}_\varphi(P_{\mathrm{top}})$.

Let us briefly recall the example given on \cite[p. 193]{Wa} of a homeomorphism without a measure with maximal entropy. We start describing the $\beta$-shift. Let $\beta > 1$ be given and take the expansion of $1$ in powers of $\beta^{-1}$, that is, $1 = \sum_{n\,=\,1}^{+\infty} \, a_n \,\beta^{-n}$
where
$$a_1  =  [\beta] \quad \quad \text{and} \quad \quad a_n  =  \big[\beta^n - \sum_{j\,=\,1}^{n-1} \, a_j \,\beta^{n-j}\big] \quad \quad \forall\, n \geqslant 2.$$
Then $0 \leqslant a_n \leqslant k-1$ for all $n \in \mathbb{N}$, where $k = [\beta] + 1.$ So we can consider $a=(a_n)_{n \, \in \, \mathbb{N}}$ as a point in the space $\Sigma^+_k = \prod_{i\,=\,1}^{+\infty}\,\{0, 1, \cdots, k-1\}$, within which we define the lexicographical ordering, that is, $(x_n)_{n \, \in \, \mathbb{N}} \, < \,(y_n)_{n \, \in \, \mathbb{N}}$ if $x_j < y_j$ for the smallest $j$ with $x_j \neq y_j$. Let $\sigma_+: \,\Sigma^+_k  \to \Sigma^+_k$ be the one-sided shift transformation. Note that $\sigma_+^n(a) \leqslant a$ for every $n \in \mathbb{N}_0$. Define
$$Y_\beta = \Big\{x = (x_n)_{n \, \in \, \mathbb{N}} \,\in \, \Sigma^+_k \, \colon \quad \,\sigma_+^n(x) \leqslant a \quad \forall \,n \in \mathbb{N}_0 \Big\}.$$
This is a non-empty closed subset of $\Sigma^+_k$, and one has $\sigma_+(Y_\beta) = Y_\beta$ and $\htop({\sigma_+}_{|Y_\beta}) = \log \beta$. Besides, if $\Sigma_k = \prod_{i\,=\,-\infty}^{+\infty}\,\{0, 1, \cdots, k-1\}$ and
$$X_\beta = \Big\{x=(x_n)_{n \, \in \, \mathbb{Z}} \,\in \, \Sigma_k \, \colon \quad\, (x_i, \, x_{i+1}, \cdots) \, \in \, Y_\beta \quad  \forall \,i \in \mathbb{Z}\Big\}$$
then $X_\beta$ is closed in $\Sigma_k$, invariant under the two-sided shift $\sigma$ and $\htop(\sigma_{|X_\beta}) = \log \beta$ as well.

Now choose an increasing sequence $(\beta_n)_{n \, \in \, \mathbb{N}}$ such that $1 < \beta_n < 2$ and $\lim_{n \, \to \, +\infty}\, \beta_n = 2$. Let $f_n:\, X_{\beta_n} \to X_{\beta_n}$ denote the two-sided $\beta_n$-shift and consider on $\Sigma_k$ a metric $d_n$ inducing the product topology and satisfying $d_n(x,y) \leqslant 1$ for every $x, \, y \in \Sigma_k$. Define a new space $X$ as the disjoint union of all the spaces $X_{\beta_n}$ together with a compactification point $x_\infty$, and put on $X$ the metric
$$
\rho(x,y) = \left\{\begin{array}{ll}
\frac{1}{n^2}\,d_n(x,y), & \hbox{if $x, \, y \in X_{\beta_n}$} \\
\smallskip \\
\sum_{j\, = \,n}^{p}\, \frac{1}{j^2}, & \hbox{if $x \in X_{\beta_n}$, $y \in X_{\beta_p}$ and $n < p$} \\
\smallskip \\
\sum_{j\, = \,n}^{+\infty}\, \frac{1}{j^2}, & \hbox{if $x=x_\infty$ and $y \in X_{\beta_n}$.} \\
\end{array}
\right.
$$
\smallskip

\noindent Then $(X, \,\rho)$ is a compact metric space and the sequence of subsets $\big(X_{\beta_n}\big)_{n \, \in \, \mathbb{N}}$ converges to $x_\infty$, that is, the sequence
$$n \in \mathbb{N} \quad \mapsto \quad \tau_n = \inf \big\{\rho(z,x_\infty) \, \colon\, \,z \in X_{\beta_n}\big\}$$
converges to $0$. Moreover, the map $f: X \to X$ defined as $f_{|X_{\beta_n}} = f_n$ and $f(x_\infty)=x_\infty$ is a homeomorphism of $(X,\rho)$; and the Borel $f$-invariant probability measures are given by
$$\sum_{n \, = \, 1}^{+\infty}\, p_n \,\mu_n + \Big(1-\sum_{n \, = \, 1}^{+\infty}\, p_n\Big)\,\delta_{x_\infty}$$
where $\mu_n \in \mathcal{P}_{f_n}(X_{\beta_n})$ for every $n \in \mathbb{N}$, and the numbers $p_n$ are non-negative and satisfy $\sum_{n \, = \, 1}^{+\infty}\, p_n \leqslant 1$. Hence the ergodic elements of $\mathcal{P}_f(X)$ are either ergodic measures in $\mathcal{P}_{f_n}(X_{\beta_n})$ for some $n$ or $\delta_{x_\infty}$. Therefore, if $\mathcal{E}_f(X)$ stands for the subset of ergodic measures in $\mathcal{P}_f(X)$, then
\begin{eqnarray*}
\htop(f) &=& \sup\,\big\{h_\mu(f): \,\mu \in \mathcal{E}_f(X)\big\} \,=\, \sup_{n \, \in \,\mathbb{N}}\,\sup \, \big\{h_{\mu_n}(f_n): \,\mu_n \in {\mathcal{E}}_{f_n}(X_{\beta_n})\big\} \\
&=& \sup_{n \, \in \,\mathbb{N}}\,\htop(f_n) \,=\, \lim_{n \, \to \, +\infty}\, \log \beta_n \,=\, \log 2.
\end{eqnarray*}
Now, if $f$ had a maximal entropy measure, then there should exist an ergodic maximal entropy measure $\mu$. Thus $\mu$ would belong to $\mathcal{E}_{f_n}(X_{\beta_n})$ for some $n$, and so $h_\mu(f) = \log \beta_n < \log 2.$

\smallskip

Let us look instead for a maximizing probability measure of $h^*$.

\begin{lemma}\label{le:delta-infinito}
Let $\varepsilon >0$ be given and, for each $n \in \mathbb{N}$, consider $\mu_n \in  {\mathcal{P}}_{f_{n}}(X_{\beta_n})$ such that $\htop(f_n) = h^*_{\mu}(f_n)$. Then any accumulation point of $(\mu_n)_{n \, \in \, \mathbb{N}}$ in the weak$^*$ topology is $\delta_{x_\infty}$.
\end{lemma}

\begin{proof} Take $\psi \in C(X)$. Our aim is to show that $\lim_{n \,\to \,+ \infty}\,\int \psi \, d\mu_n = \psi(x_\infty)$. As $\psi$ is continuous on the compact $X$, the subsets $\big(X_{\beta_n}\big)_{n \, \in \, \mathbb{N}}$ are pairwise disjoint and converge to $x_\infty$ with respect to the metric $\rho$, then the sequence of continuous maps $\big(\psi_n = \psi_{|X_{\beta_n}}\big)_{n \,\in \,\mathbb{N}}$ converges uniformly to $\psi(x_\infty)$. Therefore,
\begin{eqnarray*}
\left|\int \psi \, d\mu_n - \psi(x_\infty)\right|
\,=\, \left|\int \big(\psi_n  - \psi(x_\infty)\big) \, d\mu_n\right| \,\leqslant\, \|\psi_n  - \psi(x_\infty)\|_\infty \quad \,\stackrel{n\, \to \, +\infty}{\longrightarrow} \,\,0.
\end{eqnarray*}
\end{proof}

From Lemma~\ref{le:delta-infinito} and the upper semi-continuity of $h^*$ we also conclude that
$$h^*_{\delta_{x_\infty}}(f) \,\geqslant\, \limsup_{n \,\to \, +\infty} \, h^*_{\mu_n}(f) \,=\, \limsup_{n \,\to \, +\infty} \, h^*_{\mu_n}(f_n) \,=\, \log 2.$$
Since by definition $h^*_{\delta_{x_\infty}}(f) \leqslant \mathfrak{h}_{\delta_{x_\infty}}(f) \leqslant \htop(f) = \log 2$, the measure
$\delta_{x_\infty}$ maximizes both $h^*$ and $\mathfrak{h}$. On the contrary, $h_{\delta_{x_\infty}}(f) = 0.$
In particular, as $h_{\delta_{x_\infty}}(f)  \, < \, \mathfrak{h}_{\delta_{x_\infty}}(f)$, the measure $\delta_{x_\infty}$ belongs to $\mathscr D$ (cf. \eqref{def:D}).

\subsection{Pressure derived from Ruelle-Perron-Frobenius transfer operators}\label{sec:RPF-bounded}

Some of the statistical properties of equilibrium states are often proved using transfer operators, and the topological pressure arises as the logarithm of the spectral radius of such an operator (cf. \cite{Ke}). However, the spectral theory of these operators is more powerful when the transfer operator preserves the spaces of H\"older continuous or bounded variation potentials. In what follows, we recall some of these concepts and show that Theorem~\ref{thm:main} also imparts a new insight in the thermodynamic formalism of piecewise continuous maps.

Let $f: X\to X$ be a piecewise continuous map on a metric space $(X,d)$ and assume that $\kappa \,:= \,\sup_{x \, \in \, X}\,\#f^{-1}(x) < +\infty$. Then, given a potential $\varphi \in B_m(X)$, the \emph{Ruelle-Perron-Frobenius transfer operator with weight $\varphi$} is (well) defined by
\begin{eqnarray}\label{def:transfer}
\mathcal{L}_\varphi \colon \quad B_m(X) \quad &\to& \quad B_m(X) \nonumber \\
\psi \quad &\mapsto& \quad \mathcal{L}_\varphi (\psi) \colon \,\, x \in X \,\mapsto \,\sum_{f(y)\,=\,x} e^{\varphi (y)} \, \psi(y).
\end{eqnarray}
Denote by $r(\mathcal{L}_\varphi)$ the spectral radius of $\mathcal{L}_\varphi$ which, according to Gelfand's formula (cf. \cite{DS-2}),
may be computed by $r(\mathcal{L}_\varphi) = \lim_{n\,\to\, +\infty} \,\sqrt[n]{\|\mathcal{L}_\varphi^n\|}$.

\begin{lemma}\label{le:piecewisespectral}
The function $P: B_m(X) \to \mathbb R$ given by $P(\varphi)=\log r(\mathcal{L}_\varphi)$ is a pressure function.
\end{lemma}

\begin{proof}
Fix $\varphi \in B_m(X)$. Since the space $B_m(X)$ is endowed with the supremum norm and $\mathcal{L}_\varphi$ is a positive operator, for every $n \in \mathbb{N}$ one has
$$\|\mathcal{L}_\varphi^n\| = \sup_{\|\psi\|_\infty \, =\,1}\, \|\mathcal{L}^n_\varphi(\psi)\|_\infty = \|\mathcal{L}^n_\varphi(\mathbf 1)\|_\infty.$$
So $r(\mathcal{L}_\varphi)=\lim_{n\,\to\, +\infty} \,\sqrt[n]{\|\mathcal{L}_\varphi^n(\mathbf 1)\|_\infty}$, which is bounded by $\kappa \, e^{\|\varphi\|_\infty}$.
Now, given $a \in [0,1]$, $\varphi_1, \,\varphi_2\in B_m(X)$ and $n \in \mathbb{N}$, we write
\begin{align*}
\mathcal{L}^n_{a\varphi_1 + (1-a)\varphi_2}(\mathbf 1)(x) = \sum_{f^n(y)=x} e^{S_n(a\varphi_1 + (1-a)\varphi_2) (y)}
	= \!\!\sum_{f^n(y)=x}  \big(e^{S_n\varphi_1(y)}\big)^a \big(e^{S_n\varphi_2(y)}\big)^{1-a}
\end{align*}
and apply H\"older's inequality to get
$$\|\mathcal{L}^n_{a\varphi_1 + (1-a)\varphi_2}(\mathbf 1)\|_\infty \leqslant \|\mathcal{L}^n_{\varphi_1}(\mathbf 1)\|_\infty^a  \|\mathcal{L}^n_{\varphi_2}(\mathbf 1)\|_\infty^{1-a}.$$
Taking logarithm, dividing by $n$ and letting $n$ go to $+\infty$, we obtain
$$\log \big(r(\mathcal{L}_{a\,\varphi_1 + (1-a)\,\varphi_2})\big) \,\leqslant \,a\, \log \,\big(r(\mathcal{L}_{\varphi_1})\big) + (1-a)\, \log \,\big(r(\mathcal{L}_{\varphi_2})\big)$$
thereby showing the convexity of the function $P$. The monotonicity follows from the positivity of the operator $\mathcal{L}_\varphi$ and the proof of the translation invariance is immediate.
\end{proof}

Consequently, Theorem~\ref{thm:main} yields the following variational principle.

\begin{corollary}\label{corPE}
Let $f: X \to X$ be a piecewise continuous map on a metric space $X$ such that $\kappa \,:= \,\sup_{x \, \in \, X}\,\#f^{-1}(x) < +\infty$. Given $\varphi \in B=B_m(X)$, there exists an upper semi-continuous map $\mathfrak h_B\colon \cP_a(X) \to\mathbb R$ such that
$$
\log r(\mathcal{L}_\varphi) \, = \,\max_{\mu\,\in\, \cP_a(X)} \,\Big\{\mathfrak h_B(\mu) + \int \varphi \, d\mu\Big\}
$$
In particular, there is $\mu_\varphi \in \cP_a(X)$ satisfying Rohlin's-like formula (cf. \cite{Rohlin})
$$\mathfrak h_B(\mu_\varphi) = \int \,\log \,\big(r(\mathcal{L}_\varphi)\, e^{-\varphi}\big) \, d\mu_\varphi.$$
\end{corollary}

We illustrate this result with the following class of examples.

\begin{example}
Consider a $C^1$-piecewise expanding map $f:X \to X$ whose domain is the union of a finite number of subintervals $X_i = [a_i,b_i)$ or $X_i = (a_i,b_i]$, where $a_i < b_i$, within which $f$ is continuous. Let $\varphi = -\log |f'|$, which we assume to be piecewise continuous and bounded, though it may not exhibit any further regularity. The corresponding transfer operator is given by
$$\psi \in B_m(X) \quad \mapsto \quad \mathcal{L}_\varphi (\psi)(x) = \sum_{f(y)\,=\, x} \frac1{|f^\prime(y)|} \, \psi(y)$$
and Corollary~\ref{corPE} ensures that there exists $\mu_\varphi \in \cP_a(X)$ such that
$$\mathfrak h_B(\mu_\varphi) + \int \log |f'| \, d\mu_\varphi = \log r(\mathcal{L}_\varphi).$$
For instance, the Lorenz maps satisfy the previous assumptions with $X =[-1,0) \cup (0, 1]$.
\end{example}

We observe that, while dealing with a transfer operator acting on a suitable Banach space $\mathcal X $
and exhibiting a spectral gap, Giulietti et al \cite[Theorem~F]{GKLM} showed a variational principle similar to the one
in Corollary~\ref{corPE} with an entropy-like function $\mathbf h_{\mathcal X}$ computed by
$$
\mathbf h_{\mathcal X}(\mu)=\inf_{\phi \,\in\, \mathcal X} \Big\{ \log \lambda_{\mathcal X}(\phi) - \int \phi\, d\mu\Big\}
$$
for every $f$-invariant probability measure $\mu$, where $\lambda_{\mathcal X}(\phi)$ denotes
the spectral radius of the transfer operator $\mathcal L_\phi : \mathcal X \to \mathcal X$.
In general, since $\mathcal X \subsetneq B_m(X)$ one has $h_\mu(f) \leqslant \mathfrak h_B(\mu) \leqslant \mathbf h_{\mathcal X}(\mu)$ for every $\mu\,\in\, \cP_f(X)$.
In the special case that $f$ is a Ruelle expanding map on a compact metric space $X$
and $\mathcal X = C^{\alpha}(X)$, $\alpha>0$, the spectral radius of the operator $\mathcal L_\phi$ acting on both spaces
$C^{\alpha}(X)$ and $C(X)$ coincide and the three notions of entropy (with $B=C(X)$) are the same.

\subsection{Finitely additive equilibrium states and second order phase transition}\label{sec:second-order}

Consider the Manneville-Pomeau family of maps $f_{\alpha}:[0,1]\to [0,1]$, $\alpha>0$, given by
\begin{equation}\label{eq:LSV}
f_{\alpha}(x)= \left\{\begin{array}{ll}
x\,(1+2^\alpha x^{\alpha}) & \quad \text{if } x \,\in\, [0,\frac12[ \\[0.2cm]
2x-1 & \quad \text{if } x\,\in \,[\frac12,1].
\end{array}
\right.
\end{equation}
It is known that the metric entropy map of each $f_\alpha$ is upper semi-continuous, and that this family exhibits phase transitions with respect to the potentials $\varphi_{\alpha,t} = -t\,\log \,|f_\alpha^{\prime}|$, parameterized by $t \in \mathbb R$. For instance:
\smallskip

\begin{enumerate}
\item[(MP$_1$)] If $\alpha>0$ and $t\in \,\,]-\infty,1[$, there exists a unique equilibrium state $\mu_{\alpha,t}$ for
$f_\alpha$ and $\varphi_{\alpha,t}$.
\medskip
\item[(MP$_2$)] If $\alpha>0$, the map $t \in [1,+\infty[ \,\,\mapsto \,\Ptop(-t\,\log\, |f_\alpha'|)$ is equal to zero
and the Dirac measure $\delta_0$ is an equilibrium state with respect to $\varphi_{\alpha,t}$.
\medskip
\item[(MP$_3$)] If $0 < \alpha < 1$, there exist two equilibrium states with respect to $\varphi_{\alpha,1}$, namely the Dirac $\delta_0$ and an $f_\alpha$-invariant {probability} measure $\mu_{\alpha,1}$ which is absolutely continuous with respect to the Lebesgue measure; moreover,
the map $t \in \mathbb R \, \mapsto \,\Ptop(-t\,\log\, |f_\alpha'|)$ is not $C^1$.
\medskip
\item[(MP$_4$)] If $\alpha \geqslant 1$, there exists a unique equilibrium state for $\varphi_{\alpha,t}$ for any $t\in \mathbb R$; moreover, the map $t \in \mathbb R\, \mapsto \,\Ptop(-t\,\log\, |f_\alpha'|)$ is $C^1$, but not $C^2$,
and there is an $f_\alpha$-invariant, $\sigma$-finite and infinite measure which is absolutely continuous with respect to the Lebesgue measure.
\end{enumerate}
\smallskip
We refer the reader to \cite{BTT,Lo,Ve} for an ample discussion on phase transitions of the Manneville-Pomeau family.
\medskip

It is worth noticing that the $C^1$-smoothness of the pressure is compatible with the presence of second-order phase transitions. Yet, in the special case of the Manneville-Pomeau maps and $\alpha\geqslant 1$, the second order phase transitions for the potential $\varphi_{\alpha,t} \in C^\alpha([0,1]) \subset C^0([0,1])$ can be detected as first order phase transitions when one considers the pressure function defined on the space of bounded observables or, equivalently, when one replaces the space of invariant probability measures by the larger set of invariant finitely additive measures. More precisely, let $P(f_\alpha, \cdot): B_m([0,1]) \to \mathbb R$ be the pressure function defined by $P(f_\alpha, \varphi)=\log r(\mathcal{L}_\varphi)$ (recall Lemma~\ref{le:piecewisespectral} and the definition of the transfer operator
$\mathcal{L}_\varphi$ in ~\eqref{def:transfer}). Then
\begin{equation}\label{eq:VPMP001}
P(f_\alpha, \varphi) \,\,= \,\,\max_{\mu \,\in \,\mathcal P_a([0,1])} \,\Big\{ \mathfrak{h}(\mu) + \int \varphi \, d\mu \Big\}
\end{equation}
where
$$
\mathfrak{h}(\mu)=\inf_{\varphi\,\in \,B_m([0,1])} \, \,\Big\{ \log r(\mathcal{L}_\varphi) - \int \varphi \,d\mu\Big\}
$$
for every finitely additive Borel probability measure $\mu$. As $P(f_\alpha, \log |f_\alpha'|) = \log r(\mathcal{L}_{-\log |f_\alpha'|}) = 0$ and
$$
\|\mathcal{L}^n_\varphi\|_\infty \geqslant \|\mathcal{L}^n_\varphi (\mathbf 1)\|_\infty \geqslant (\mathcal{L}^n_\varphi (\mathbf 1))(0) \geqslant e^{n \varphi(0)}
$$
for every $\varphi\in B_m([0,1])$, one concludes that $\mathfrak{h}(\delta_0)\geqslant 0$. Moreover, for every $f_\alpha$-invariant probability measure $\mu$,
\begin{align*}
\mathfrak{h}(\mu)
	& =\inf_{\varphi\,\in \,B_m([0,1])} \, \,\Big\{\log r(\mathcal{L}_\varphi) - \int \varphi \,d\mu\Big\}
	\, \leqslant \,\inf_{\varphi\,\in \,C^0([0,1])} \, \,\Big\{\log r(\mathcal{L}_\varphi) - \int \varphi \,d\mu\Big\} \\
	& = \inf_{\varphi\,\in \,C^0([0,1])} \, \,\Big\{P(f_\alpha,\varphi) - \int \varphi \,d\mu\Big\}
	\,\,=\,\, h_\mu(f_\alpha)
\end{align*}
where the last equality is due to the fact that the metric entropy map of $f_\alpha$ is upper semi-continuous, while the last but one is a consequence of the conjugation between $f_\alpha$ and the doubling map, 
for which such an equality holds (see e.g. \cite{CRV-Rec}). Therefore, $\mathfrak{h}(\delta_0) \leqslant h_\mu(\delta_0)=0$; hence $\mathfrak{h}(\delta_0) = 0$. In the particular case of $\varphi = -\log |f_\alpha'|$ we get
\begin{align*}
P(f_\alpha, -\log |f_\alpha'|) \,\,=\,\, 0
	\,= \max_{\mu \,\in \,\mathcal P_a([0,1])} \,\Big\{ \mathfrak{h}(\mu) - \int \log |f_\alpha'| \, d\mu \Big\}
\end{align*}
and the Dirac measure $\delta_0$ attains the maximum.

\medskip

A second equilibrium state appears while one looks for absolutely continuous finitely additive invariant measures. Consider the Lebesgue measure on $[0,1]$ (which we abbreviate into $\text{Leb}$) and the Banach space $L^\infty(\text{Leb}, [0,1])$ of equivalent classes of the essentially bounded real-valued Lebesgue measurable maps $\varphi: [0,1] \to \mathbb{R}$, under the relation $\varphi \backsim \psi$ if and only if $\varphi = \psi$ at $\text{Leb}$ almost everywhere. Endow $L^\infty(\text{Leb}, [0,1])$ with the essential supremum norm. The Hewitt-Yosida representation theorem (cf. \cite{HY}) informs that the dual of $L^\infty(\text{Leb}, [0,1])$ is the space of bounded finitely additive measures $m$ on the Borel subsets of $[0,1]$ which are absolutely continuous with respect to  $\text{Leb}$ (in the sense that $\text{Leb}(A) = 0$ implies $m(A)=0$), with the norm of total variation. Denote by $\mathcal{P}_{a, \, \text{Leb}}([0,1])$ the set of normalized elements of $\big(L^\infty(\text{Leb}, [0,1])\big)^*$ and by $\mathcal{P}_{a, \, \text{Leb}}(f_\alpha) \subset \mathcal{P}_{a, \, \text{Leb}}([0,1])$ the subset of those elements which are $f_\alpha$-invariant. According to \cite{CER}, the space $\mathcal{P}_{a, \, \text{Leb}}(f_\alpha)$ is non-empty.

\medskip

Given $\mu \in \big(L^\infty(\text{Leb}, [0,1])\big)^*$, consider the upper semi-continuous entropy function
\begin{equation}\label{def:Bakhtin-def}
H(\mu) = \left\{\begin{array}{ll}
\inf_{\varphi \, \in \, L^\infty(\text{Leb}, [0,1])} \quad \big\{\log r(\widehat{\mathcal{L}_\varphi})  - \int \varphi \,d\mu\big\}
& \quad \text{ if $\mu$ is $f_\alpha$-invariant}\\
-\infty & \quad \text{ otherwise}
\end{array}
\right.
\end{equation}
where $r(\widehat{\mathcal{L}_\varphi})$ is the spectral radius of a transfer operator $\widehat{\mathcal{L}_\varphi}$ obtained as the extension to $L^\infty(\text{Leb}, [0,1])$ of the operator $\mathcal{L}_\varphi$ acting on $B_m([0,1])$ (for more details we refer the reader to \cite{Bak}). The previous definition allows us to build a pressure function on $L^\infty(\text{Leb}, [0,1])$. Given $\varphi \in L^\infty(\text{Leb}, [0,1])$, define
\begin{equation}\label{eq:defptil}
\widehat{P}(f_\alpha, \, \varphi) \,\,=\,\, \sup_{\mu \, \in \, \mathcal{P}_{a, \, \text{Leb}}(f_\alpha)}\,\Big\{H(\mu) + \int \varphi \,d\mu\Big\}.
\end{equation}
By the Legendre-Fenchel duality, 
$$\widehat{P}(f_\alpha, \, \varphi) \,\,=\,\, \log r(\widehat{\mathcal{L}_\varphi}).$$
Moreover, $\widehat{P}$ satisfies the following variational principle (cf. \cite[Theorem 1]{Bak})
\begin{equation}\label{ig:defptil}
\widehat{P}(f_\alpha, \, \varphi) \,\, = \max_{\mu \, \in \, \mathcal{P}_{a, \, \text{Leb}}(f_\alpha)}\,\Big\{ H(\mu) + \int \varphi \,d\mu\Big\}.
\end{equation}
In addition, as the spectral radius can be evaluated using the constant map $\textbf{1}$, which belongs to both $B_m([0,1])$ and $L^\infty(\text{Leb}, [0,1])$, one has
$$\log r(\widehat{\mathcal{L}_{-\log|f'_\alpha|}}) \,\,= \,\, \log r(\mathcal{L}_{-\log|f'_\alpha|}) \,\,=\,\, 0.$$
Consequently,
$$0 \,\,=\,\, P(f_\alpha,-\log|f'_\alpha|) \,\,=\,\, \widehat{P}(f_\alpha,-\log|f'_\alpha|)	
 	\,\,= \max_{\mu \,\in\, \mathcal P_{a, Leb}(f_\alpha)} \,\Big\{H(\mu) - \int \log|f'_\alpha| d\mu\Big\}.$$

\medskip

Let $m_\alpha$ be a finitely additive probability measure absolutely continuous with respect to $\text{Leb}$ where the previous maximum is attained, that is,
$$H(m_\alpha) \,\,= \,\ \int  \log |f^\prime_\alpha| \, dm_\alpha.$$
Then
\begin{eqnarray*}
0 \,\,&=&\,\, \widehat{P}(f_\alpha,-\log|f'_\alpha|) \,\, = \,\,H(m_\alpha) - \int  \log |f^\prime_\alpha| \, dm_\alpha \\
&=& \inf_{\varphi\, \in\, L^\infty(Leb, [0,1])}\, \Big\{\log r(\widehat{\mathcal{L}_\varphi}) -  \int \varphi \,d m_\alpha\Big\}  - \int  \log |f^\prime_\alpha| \, dm_\alpha \\
& \leqslant	& \inf_{\varphi \,\in \,B_m([0,1])} \,\Big\{\log r(\mathcal L_\varphi) -  \int \varphi \,d m_\alpha\Big\} - \int  \log |f^\prime_\alpha| \, dm_\alpha \\
&=& \mathfrak{h}(m_\alpha) - \int  \log |f^\prime_\alpha| \, dm_\alpha  \,\,\leqslant\,\, P(f_\alpha,-\log|f'_\alpha|) \,\,=\,\,0.
\end{eqnarray*}

\noindent Thus, summoning {the previous information on $m_{\alpha}$ and $\delta_0$}, and Corollary~\ref{thm:Gateaux}, one deduces that:

\begin{corollary}
If $(f_\alpha)_{\alpha \, \geqslant  \, 1}$ is the Manneville-Pomeau family \eqref{eq:LSV}, then $\#\,\mathcal T_{\varphi_{\alpha,1}}(\Ptop) \geqslant 2$. In particular, for every $\alpha\geqslant 1$ the map
$\Ptop: \,C^0([0,1]) \to \mathbb R$ is Gateaux differentiable, though its extension to $B_m([0,1])$ is not Gateaux differentiable at $\varphi_{\alpha,1}$.
\end{corollary}

\section{Non-additive sequences of continuous potentials: Proof of Theorem~\ref{thm:non-additive}}\label{se:thermo-formalism-0}
In this section, we fix a continuous endomorphism of a compact metric space $X$ and, instead of the classic pressure function with respect to a given potential $\varphi:\,X \to \mathbb{R}$ and the sequence of sums $(S_n \varphi)_{n \, \in \, \mathbb{N}}$, we consider non-additive sequences of continuous potentials. Although these objects arise naturally in the study of Lyapunov exponents and dimension theory, the non-additive thermodynamic formalism is still barely understood.
We refer the reader to \cite{Ba96,Ba06,FK12,Noe} for a thorough discussion on these topics.

Let $f: X \to X$ be a continuous map on a compact metric space $(X,d)$. We say that a sequence $\Phi=(\varphi_n)_{n \, \in \, \mathbb{N}} \in C(X)^\mathbb{N}$ of continuous potentials is
\medskip

\begin{enumerate}
\item \emph{sub-additive} if $\, \varphi_{m+n}\,\leqslant\, \varphi_{m}+\varphi_{n}\circ f^m \quad  \,\forall \,m,\,n\, \in \, \mathbb{N}$;
\medskip
\item \emph{almost additive} if there exists a uniform constant $C>0$ such that
$$\varphi_{m}+\varphi_{n}\circ f^m-C \,\leqslant\, \varphi_{m+n}\,\leqslant \, \varphi_{m}+\varphi_{n}\circ f^m+C \quad \quad \forall \,m,\,n\, \in \, \mathbb{N};$$
\item \emph{asymptotically additive} if for any $\varepsilon > 0$ there exists $\varphi_\varepsilon \in C(X)$ such that
$$\limsup_{n\,\to\, +\infty} \,\frac1n \,\Big\| \varphi_n - \sum_{j=0}^{n-1}\,\varphi_\varepsilon \circ f^j \Big\|_\infty < \varepsilon.$$
\end{enumerate}

It is known that every almost additive sequence is asymptotically additive, and that for every asymptotically additive sequence $\Phi=(\varphi_n)_{n \, \in \, \mathbb{N}}  \in C(X)^\mathbb{N}$ there exists $\varphi \in C(X)$ such that
$$\limsup_{n\,\to\, +\infty} \,\frac1n\, \Big\|\varphi_n -\sum_{j=0}^{n-1}\,\varphi \circ f^j\Big\|_\infty=0$$
(cf. \cite{Noe,FH}). Therefore, both the variational principle and the existence of finitely additive equilibrium states established in Theorem~\ref{thm:second-main} admit an immediate generalization to this context
(the modifications necessary to deal with sequences in other Banach spaces are left as an easy exercise to the interested reader.)
 Henceforth, we will aim at the more general context of sub-additive sequences of continuous potentials.

\begin{definition}\label{def:press}
Given a sequence $\Phi=(\varphi_n)_{n \, \in \, \mathbb{N}} \subset C(X)$ of continuous potentials, the \emph{non-additive topological pressure} is defined by
\begin{equation}\label{eq:P-non-additive}
P(f, \,\Phi) = \lim_{\vep\,\to\,0^+}\, \limsup_{n\,\to\,+\infty}\, \frac1n \,\log \Big(\sup_E \,\sum_{x\,\in\, E}\, e^{\varphi_n(x)}\Big)
\end{equation}
where the supremum is taken over the $(n,\vep)$-separated subsets $E$ of $X$.
\end{definition}

This definition coincides with the usual notion of topological pressure $P_{\text{top}}(f,\varphi)$ when there is $\varphi \in C(X)$ such that $\varphi_n = \sum_{j=0}^{n-1}\,\varphi\circ f^j$ for every $n \in \mathbb{N}$. It is known (cf. \cite{Ba06}) 
that every almost additive sequence of continuous potentials which have bounded distortion admits a unique equilibrium state, which is a Gibbs measure. More recently, it was proved in \cite{Noe} that any almost additive or asymptotically additive sequence of continuous potentials have the same pressure of an additive sequence associated to a continuous potential. However, it is not known whether this potential inherits the distortion properties of the original almost additive sequence. Moreover, for sub-additive sequences of continuous potentials no general construction of equilibrium states is known, though for these sequences Cao, Feng and Huang~\cite{CFH08} established the following general variational principle.

\begin{theorem}\cite{CFH08}\label{thm:Chong} Given a continuous self-map $f: X \to X$ of a compact metric space $X$, if $\Phi=(\varphi_n)_{n \, \in \, \mathbb{N}}$ is a sub-additive sequence of continuous potentials such that
$P(f, \,\Phi) > -\infty$, then
\begin{equation}\label{eq:Chong}
P(f, \,\Phi) \,= \sup_{\mu \,\in \,\mathcal P_f(X)}\,\Big\{h_\mu(f) +
\mathcal F_*(\Phi,\,\mu) \Big\}
\end{equation}
where, for every $f$-invariant probability measure $\mu$,
$$\cF_*(\Phi,\,\mu)\,:=\, \lim_{n\, \to \, + \infty}\, \frac1n \,\int \varphi_n \,d\mu.$$
\end{theorem}

We note that if $\Phi=(\varphi_n)_{n \, \in \, \mathbb{N}}$ is a sub-additive sequence of continuous potentials then $\cF_*(\Phi,\,\mu)$ is well defined for every $\mu \in \cP_f(X)$. In fact, the sequence of real numbers $(a_n)_{n\,\in\, \mathbb R}$ given by $a_n = \int \varphi_n \, d\mu$ is sub-additive, hence $\lim_{n\,\to\,+\infty}\, \frac1n \, a_n = \inf_{n \,\in\, \mathbb{N}}\, \frac1n \,a_n$ by Fekete's Lemma.

\subsection{An alternative variational principle for sub-additive sequences}
One might expect a counterpart of Theorem~\ref{thm:main} for a more general context of Banach spaces of sequences of functions. This faces non-trivial difficulties, though. Contrary to what happens within the simpler case of almost additive sequences, albeit providing a convex cone in the space of sequences of potentials, sub-additivity is not preserved under multiplication by negative numbers. This is a major obstruction since the entropy in Theorem~\ref{thm:main} is defined using observable maps $\varphi$ such that $-\varphi$ has non-positive pressure, which makes the Banach space generated by sub-additive sequences not suitable to this approach. To overcome this difficulty we will combine two pressure functions to which we apply Theorem~\ref{thm:main}, thus reducing the previous problem, concerning sub-additive
sequences of functions, to the construction of equilibrium states for a single bounded potential
(see Lemma~\ref{le:sub-sdditive}).

In order to apply Kingman's Sub-additive Ergodic Theorem, we need to narrow our analysis to the set
$$\mathcal S_b = \Big\{\Phi = (\varphi_n)_{n\,\in\, \mathbb N} \in C(X)^{\mathbb N} \colon  \Phi  \; \text{is sub-additive and }
\!\! \inf_{x\, \in \,X} \,\Big[\inf_{n\,\in\, \mathbb{N}}\, \frac1n \,\varphi_n(x) \Big] > -\infty\Big\}.$$
This set comprises relevant families of sequences of continuous potentials arising within the theory of linear cocycles, as we will detail on Subsection~\ref{subsec:cocycles}. We also observe that $\Phi \in \mathcal S_b$ if and only if it is sub-additive and $\cF_*(\Phi,\mu)>-\infty$ for every $f$-invariant probability measure $\mu$ (cf. \cite[pp. 336--337]{Sc98}). Moreover:

\begin{lemma}\label{le:sub-sdditive} Given $\Phi = (\varphi_n)_{n \, \in \, \mathbb{N}}  \in  \,\mathcal S_b$, the map $\psi_\Phi$ defined by
$$ x \,\in\, X \quad \mapsto \quad \psi_\Phi(x) = \inf_{n\, \in \,\mathbb{N}}\, \frac1n \,\varphi_n(x)$$
is measurable, upper semi-continuous, belongs to $B_m(X)$ and satisfies
$$\int \psi_\Phi \, d\mu  =  \cF_*(\Phi,\mu)  \quad \quad \forall\, \mu\, \in\,  \cP_f(X).$$
\end{lemma}

\begin{proof}
For every $\Phi = (\varphi_n)_{n \, \in \, \mathbb{N}}  \in  \,\mathcal S_b$, the corresponding map $\psi_\Phi$ is measurable and upper semi-continuous, hence upper bounded on the compact $X$. Since $\Phi$ belongs to $\mathcal S_b$, the map $\psi_\Phi$ is also lower bounded. Moreover, by the Kingman's Sub-additive Ergodic Theorem (see \cite{Wa}), 
the maps $\inf_{n\, \in \,\mathbb{N}}\, \frac1n \,\varphi_n$ and $\liminf_{n\, \to +\infty}\, \frac1n \,\varphi_n$ coincide in a set with full measure and, for every $\mu\, \in\,  \cP_f(X)$ one has
$\int \,\liminf_{n\, \to +\infty}\, \frac1n \,\varphi_n \,\, d\mu  = \cF_*(\Phi,\mu).$
So, $\int \psi_\Phi \, d\mu  =  \cF_*(\Phi,\mu)$ as well.
\end{proof}

In the remaining of this subsection we restrict to $\mathcal S_b$ and consider $P(f, .)$ in order to improve the variational relation \eqref{eq:Chong}.
More precisely, we will show the following counterpart of
Theorem~\ref{thm:second-main} in this context, using $B_m(X)$ instead of $C(X)$.

\begin{theorem}\label{cor:additive-2}
Let $f$ be a continuous self-map of a compact metric space $(X,d)$ whose topological entropy is finite.
Then there exists an affine and upper semi-continuous entropy map ${\mathfrak h}_1: \mathcal{P}_a(X) \to \mathbb{R}$ such that
\begin{equation*}\label{eq:new-vp}
P(f, \,\Phi) =  \max_{\mu\, \in\, \mathcal P_a(X)}\, \Big\{{\mathfrak h}_1(\mu) + \int \psi_\Phi \, d\mu\Big\}
	\quad \quad \forall \,\Phi \in \mathcal S_b.
\end{equation*}
\end{theorem}

\begin{proof}
An effortless computation shows that the map $\Gamma_1: \,B_m(X) \to \mathbb{R}$ defined by
$$\psi \in B_m(X) \quad \mapsto \quad \Gamma_1(\psi) = \sup_{\mu \,\in \,\mathcal P_f(X)}\,\Big\{h_\mu(f) + \int \psi \, d\mu \Big\}$$
is a pressure function. Therefore, we may apply Theorem~\ref{thm:main} and conclude that there exists an affine and upper semi-continuous map ${\mathfrak h}_1: \mathcal{P}_a(X) \to \mathbb{R}$ such that, for every $\mu \in \mathcal{P}_a(X)$,
$${\mathfrak h}_1(\mu) =  \inf_{\psi \,\in \,\mathcal{A}_{\Gamma_1}}\, \int \psi \, d\mu$$
and, for each $\psi\in B_m(X)$,
$$
\Gamma_1(\psi) = \max_{\mu \,\in \,\mathcal P_a(X)}\,\Big\{{\mathfrak h}_1(\mu) + \int \psi \, d\mu \Big\}.
$$
Besides, $P(f,\Phi) \,\geqslant \, \inf_{\mu\, \in\,  \cP_f(X)} \, \cF_*(\Phi,\mu) \,>\,-\infty$
for every $(\varphi_n)_{n \, \in \, \mathbb{N}}  \in  \,\mathcal S_b$ (cf. \cite{CFH08}).
In addition, from Lemma~\ref{le:sub-sdditive} and Theorem~\ref{thm:Chong} one deduces that, for every $\Phi = (\varphi_n)_{n \, \in \, \mathbb{N}}  \in  \,\mathcal S_b$,
$$P(f,\Phi) = \Gamma_1(\psi_\Phi).$$
That is,
\begin{align*}
\sup_{\mu \,\in \,\mathcal P_f(X)} \Big\{h_\mu(f) + \mathcal F_*(\Phi,\mu) \Big\}
& = \sup_{\mu \,\in \,\mathcal P_f(X)} \Big\{h_\mu(f) + \mathcal \int \psi_\Phi  d\mu \Big\} \\
& = \max_{\mu \,\in \,\mathcal P_a(X)} \Big\{{\mathfrak h}_1(\mu) + \int \psi_\Phi  d\mu \Big\}.
\end{align*}
\end{proof}

\subsection{Application to linear cocycles and Lyapunov equilibrium states}\label{subsec:cocycles}

Non-uniform hyperbolicity is defined in terms of \textit{Lyapunov exponents}: a diffeomorphism is non-uniformly hyperbolic if it has no zero Lyapunov exponents and there exists at least one positive and one negative exponent. These numbers measure the exponential asymptotic rates of contraction or expansion along fixed directions, and became a fundamental tool to characterize chaotic dynamics. Linear cocycles turn to be a powerful mean to attest the abundance of non-uniformly hyperbolic behavior, as it allows to detach the underlying dynamics from the action it induces on a vector space. Here we are mainly interested in the existence of Lyapunov equilibrium states for linear cocycles. Some recent contributions on this topic comprise \cite{BKM, BM, FH,FK12}.

\subsubsection{\textbf{\emph{Lyapunov exponents}}}
\label{def-Lexponents}
We start recalling some preliminary notions. Let $f$ be a continuous map on a compact metric space $(X,d)$. Given an integer $\ell\geqslant 1$, a field $\mathbb K = \mathbb R$ or $\mathbb C$ and a measurable matrix-valued map $A: X \to GL(\ell,\mathbb K)$, the \emph{linear cocycle} generated by $A$ and driven by $f$ is the map
\begin{eqnarray*}
F_A: X \times \mathbb K^\ell \quad &\rightarrow& \quad  X \times \mathbb K^\ell\\
(x,v) \quad &\mapsto& \quad \left(f(x),\,A(x)v\right).
\end{eqnarray*}
Its iterates are ${F}^n_A\left(x,v\right)=\left(f^n(x),A^n(x)v\right)$, where $$A^n(x)=A(f^{n-1}(x))\cdots A(f(x))A(x)$$ for every $n \in \mathbb{N}$, $A^0(x)=x$ and, if $f$ is invertible, $$A^n(x) = A(f^{n}(x))^{-1}\cdots A(f^{-1}(x))^{-1}$$ when $n < 0$. We shall also refer to the cocycle as a pair $(f,A)$. A natural example of linear cocycle is given by the \textit{derivative cocycle} associated to a diffeomorphism $f\in \text{Diff}^{\,1}(X)$ on a compact Riemannian manifold $X$, in which case the cocycle is generated by $A(x)=Df(x)$ for each $x\in X$.

Oseledets' Theorem asserts that, under mild conditions, the Lyapunov exponents of the cocycle $(f,A)$ are well defined. More precisely, given an $f$-invariant probability measure $\mu$, if $\log\|{A}^\pm\| \in L^1(\mu)$ then
for $\mu$-almost every $x\in X$ there exist an integer $k(x)\geqslant 1$, a splitting $\mathbb K^\ell=E^{1,A}_{x}\oplus \cdots \oplus E^{k(x),A}_{x}$ and real numbers (called \emph{Lyapunov exponents})
$$\lambda_1 \left(A, \mu, x\right)> \cdots > \lambda_{k(x)}\left(A, \mu, x\right)$$
such that, for every $v\in E^{i,A}_{x} \setminus\{0\}$ and $1 \leqslant i \leqslant k(x)$,
$$A(x)\left(E^{i,A}_{x}\right) = E^{i,A}_{f(x)} \quad \quad\text{  and  } \quad \quad \lambda_i(A, \mu, x) = \lim_{n\,\to\,+ \infty}\, \dfrac{1}{n}\,\log \| A^n(x)v\|.$$
If, in addition, $\mu$ is ergodic, then $k(x)$, the Lyapunov exponents $\lambda_i(A, \mu, x)$ and the dimensions of the subspaces $E^{i,A}_{x}$ are $\mu$-almost everywhere constant, in which case one simplifies the notation by writing $\lambda_i(A,\mu)$.

\begin{remark}\label{def:Lexp}
\emph{
If $X$ is a compact Riemannian manifold, $f$ is a $C^1$-diffeomorphism on $X$ and $\mu$ is an $f$-invariant and ergodic probability measure then the Lyapunov exponents of $\mu$ are defined as the corresponding Lyapunov exponents for the derivative cocycle $A=Df$.}
\end{remark}

\subsubsection{\emph{\textbf{Singular value sub-additive potentials}}}\label{subsec:lcocs}
In what follows, $\wedge^k L$ stands for the $k$th exterior power of the linear map $L$. Assume that the linear cocycle $A: X \to GL(\ell,\mathbb K)$ is continuous. Then the Lyapunov exponents can be computed using exterior powers and a family of sub-additive sequences of potentials. More precisely, if $\mu$ is an $f$-invariant and ergodic probability measure and one takes for each $k \in \mathbb{N}$ the sub-additive sequence $\Phi_k=(\varphi_{k,n})_{n\,\in\, \mathbb{N}}$ of the continuous functions
$$x \in X \quad \mapsto \quad \varphi_{k,n}(x) = \log \| \wedge^kA^{n}(x)\|$$
then
$$\lim_{n\,\to\,+\infty} \,\frac1n \,\varphi_{k,n}(x) = \sum_{i=1}^k \,\lambda_i(A,\mu) \quad\quad \text{at $\mu$-almost every $x\in X$.}$$

Motivated by applications in dimension theory and aiming to apply their results to Falconer's singular value function and affine iterated function systems with invertible affinities, Bochi and Morris~\cite{BM} studied the following continuous parameterized family of sub-additive sequences of potentials. Given
$\vec \alpha=(\alpha_1, \alpha_2, \dots, \alpha_\ell) \in \mathbb R^\ell$ with $\alpha_1\geqslant \alpha_2 \geqslant \dots
\geqslant \alpha_\ell$, consider the sequence $\Phi_{\vec\alpha}=(\varphi_{\vec\alpha,n})_{n \,\in \,\mathbb{N}}$
defined by
$$
\varphi_{\vec\alpha,n} (x) = \log \Big( \prod_{i=1}^\ell \mathfrak s_i(A^n(x))^{\alpha_i}\Big)
$$
where $\mathfrak s_i(L)$ denotes the $i$th singular value of the linear map $L$. Then it is shown in \cite{BM} that if $\mu$ is an $f$-invariant and ergodic probability one has
\begin{equation}\label{eq:Old-VP0}
\lim_{n\,\to\,+\infty} \,\frac1n \,\varphi_{\vec\alpha,n}(x) = \sum_{i=1}^k \alpha_i \cdot \lambda_i(A,\mu) \quad \quad \text{at $\mu$-almost every $x\in X$.}
\end{equation}
Now, the variational principle for the previous family of sub-additive sequences, established by Theorem~\ref{thm:Chong}, says that, if $\htop(f)<+\infty$ and $P(f, A,\,\Phi_{\vec\alpha})$ is the pressure function defined by~\eqref{eq:P-non-additive} when $\Phi=\Phi_{\vec\alpha}$, then
\begin{equation}\label{eq:Old-VP}
P(f, A,\,\Phi_{\vec\alpha}) \,= \sup_{\mu \,\in \,\mathcal P_f(X)}\,\Big\{h_\mu(f) + \int  \sum_{i=1}^k \alpha_i \cdot \lambda_i(A,\mu,x)\,d\mu \Big\}.
\end{equation}
Invariant measures attaining the previous equality, so called \emph{Lyapunov equilibrium states}, may be hard to find. In the special context of cocycles over a full shift, the metric entropy function is upper semi-continuous and so these equilibrium states always exist. Moreover, within totally disconnected spaces, it has been shown under great generality that, for typical one-step cocycles and H\"older continuous fiber-bunched cocycles, the previous families of sequences of potentials satisfy a quasi-additivity property, and so they have unique Lyapunov equilibrium states (cf. \cite{FK12, Park}).  
More recently, Bochi and Morris (cf. \cite[Theorem~1]{BM}) proved that, in this setting, there are finitely many Lyapunov equilibrium states, these have full support, and are unique for potentials at large temperatures.

It is known that the existence of Lyapunov equilibrium states for the family $(\beta\Phi_{\vec\alpha})_{\beta\,>\,0}$ carries information on Lyapunov optimizing measures. For example, given Lyapunov equilibrium states $\mu_\beta$ with respect to $\beta\Phi_{\vec\alpha}$, $\beta > 0$, any weak$^*$ accumulation point of $(\mu_\beta)_{\beta\,>\,0}$ is an ergodic optimizing measure for the potential $\Phi_{\vec\alpha}$. Additionally, in the case of $2\times 2$ one-step dominated cocycles, Bochi and Rams~\cite{BR16} proved that Lyapunov optimizing measures always exist and, under an additional strong domination condition,
these have zero topological entropy. More information about the behavior of zero temperature limits may be read in \cite{Rz1}.

\subsection{Proof of Theorem~\ref{thm:non-additive}}

Let $f$ be a continuous map on a compact metric space $(X,d)$. For each cocycle $A \in C(X, GL(\ell,\mathbb R))$ and vector $\vec \alpha=(\alpha_1, \alpha_2, \dots, \alpha_\ell) \in \mathbb R^\ell$ with $\alpha_1\geqslant \alpha_2 \geqslant \dots \geqslant \alpha_\ell$ consider the corresponding non-additive sequence $\Phi_{\vec\alpha}$ of singular value potentials. The map
$$
\psi_{\Phi_{\vec\alpha}}(x) \,:= \,\inf_{n\, \in \,\mathbb{N}}\, \frac1n \,\varphi_{\vec\alpha,n}(x), \quad x\in X
$$
defined by Lemma~\ref{le:sub-sdditive} is a measurable, upper semi-continuous and bounded function on $X$.
Combining equations ~\eqref{eq:Old-VP0} and \eqref{eq:Old-VP} one gets
$$
P(f, A,\,\Phi_{\vec\alpha}) \,= \sup_{\mu \,\in \,\mathcal P_f(X)}\,\Big\{h_\mu(f) + \int  \psi_{\Phi_{\vec\alpha}}  \,d\mu \Big\}.
$$
Consider now the function ${\bf P}(f,\cdot )\colon \, B_m(X) \to \mathbb R$ defined by
$${\bf P}(f, \psi ) \,\,=\,\, \sup_{\mu \,\in \,\mathcal P_f(X)}\,\Big\{h_\mu(f) + \int  \psi \,d\mu \Big\}.$$
It is immediate to check that this is a pressure function. Thus, applying Theorem~\ref{thm:main}, one obtains an affine and upper semi-continuous entropy entropy function
$$\mathfrak h^{\vec\alpha} \colon \,\mathcal{P}_a(X) \,\,\to\,\, \mathbb{R} \cup \{-\infty,\, +\infty\}$$
satisfying
\begin{equation}\label{eq:equalss}
P(f, A,\,\Phi_{\vec\alpha}) \,\,=\,\, {\bf P}(f,\psi_{\Phi_{\vec\alpha}}) \,\,=\,\, \max_{\mu \,\in \,\mathcal P_a(X)}\,\Big\{\mathfrak{h}^{\vec\alpha}_\mu(f) + \int  \psi_{\Phi_{\vec\alpha}}  \,d\mu \Big\}.
\end{equation}
The set of finitely additive probability measures attaining the maximum is $f$-invariant (recall the proof of Corollary~\ref{cor:2dyn}).
Moreover, by Theorem~\ref{thm:main} there is an upper semi-continuous map ${\mathfrak h}_f: \mathcal{P}_a(X) \to \mathbb{R}$ such that, for every non-additive sequence $\Phi_{\vec\alpha}$ of singular value potentials, one has
$$P(f, A,\,\Phi_{\vec\alpha}) \,= \,\max_{\mu \,\in \,\mathcal P_a(X)}\,\Big\{{\mathfrak h}_\mu(f) + \int  \psi_{\Phi_{\vec\alpha}}\,d\mu \Big\}.$$

We are left to show that the set of finitely additive equilibrium states is non-empty for every linear cocycle in $C(X, GL(\ell,\mathbb R))$, and that the zero temperature limits of finitely additive equilibrium states have the largest value of ${\mathfrak h}^{\vec \alpha}_f$ amongst the Lyapunov optimizing measures. This is a simple consequence
of equality ~\eqref{eq:equalss} and the fact that, as
$$\sup_{\mu\,\in\, \mathcal \cP_a(X)} \, \mathfrak{h}^{\vec\alpha}_\mu(f) \leqslant \htop(f)< +\infty$$
then one has
$$
\frac1t {\bf P}(f,\,  t\Phi_{\vec\alpha}) \,=\, \max_{\mu\,\in\, \cP_a(X)}\, \Big\{\frac1t \, \mathfrak{h}^{\vec\alpha}_\mu(f) + \int \psi_{\Phi_{\vec\alpha}} \, d\mu \Big\} \quad
\underset{t\,\to \,+\infty}{\longrightarrow} \quad	\max_{\mu\,\in\, \cP_a(X)} \,\int  \psi_{\Phi_{\vec\alpha}}\, d\mu
$$
and
\begin{align*}
\frac1t P(f, A,\,t\Phi_{\vec\alpha}) \,& = \sup_{\mu \,\in \,\mathcal P_f(X)}\,\Big\{\frac1t h_\mu(f) + \int  \psi_{\Phi_{\vec\alpha}}  \,d\mu \Big\} \\
	& \underset{t\,\to\, +\infty}{\longrightarrow} \quad	
	\sup_{\mu\,\in\, \cP_f(X)} \,\int  \psi_{\Phi_{\vec\alpha}}\, d\mu  \,\,=\,\, \sup_{\mu\,\in\, \cP_f(X)} \int  \sum_{i=1}^k \alpha_i \cdot \lambda_i(A,\mu,x)\,d\mu.
\end{align*}
\smallskip

\noindent Therefore, there exists an $f$-invariant finitely additive probability measure $\mu_{\Phi_{\vec\alpha}} \in \cP_a(X)$ such that
$$\int  \psi_{\Phi_{\vec\alpha}}\, d\mu_{\Phi_{\vec\alpha}} \,\,=\,\, \sup_{\mu\,\in\, \cP_f(X)} \int  \sum_{i=1}^k \alpha_i \cdot \lambda_i(A,\mu,x)\,d\mu.$$

\subsection*{Acknowledgments}
The authors are grateful to the anonymous referees for the apposite comments and valuable suggestions that have helped us to improve the manuscript. The authors also thank G. Iommi, B. Kloeckner, A. Lopes, J. Rivera-Letelier and M. Todd for their insightful remarks on the first version of this text.  AB was partially supported by the inner Lodz University grant 11/IDUB/DOS/2021. MC, MM and PV were partially supported by CMUP, which is financed by national funds through FCT - Funda\c c\~ao para a Ci\^encia e a Tecnologia, I.P., under the project with reference UIDB/00144/2020. MC also acknowledges financial support from the project PTDC/MAT-PUR/4048/2021. PV benefited from the grant CEECIND/03721/2017 of the Stimulus of Scientific Employment, Individual Support 2017 Call, awarded by FCT, and from the project `New trends in Lyapunov exponents' (PTDC/MAT-PUR/29126/2017).


\begin{thebibliography}{98}
\providecommand{\natexlab}[1]{#1}
\providecommand{\url}[1]{\texttt{#1}}
\expandafter\ifx\csname urlstyle\endcsname\relax
\providecommand{\doi}[1]{doi: #1}\else
\providecommand{\doi}{doi: \begingroup \urlstyle{rm}\Url}\fi




\bibitem{AB06}
C. D. Aliprantis and K. C. Border.
\newblock Infinite Dimensional Analysis.
\newblock Springer-Verlag Berlin, 3rd edition, 2006.


\bibitem{ABL2011}
A. B. Antonevich, V.I. Bakhtin and A. V. Lebedev.
\newblock \emph{On t-entropy and variational principle for the spectral radii of transfer and weighted shift operators.}
\newblock Ergodic Theory Dynam. Systems 31 (2011) 995--1045.


\bibitem{AE80}
L. Asimov and A. Ellis.
\newblock Convexity Theory and its Applications in Functional Analysis.
\newblock London Math. Soc. Monographs 16, Academic Press, New York, 1980.

\bibitem{Ba}
R. Baire.
\newblock Le\c cons sur les Fonctions Discontinues, profess\'ees au Coll\`ege de France.
\newblock Gauthier-Villars, 1905.

\bibitem{BD20}
V. Baladi and M. Demers.
\newblock \emph{On the measure of maximal entropy for finite horizon Sinai Billiard maps.}
\newblock J. Amer. Math. Soc. 33:2 (2020) 381--449.

\bibitem{Bak}
V. I. Bakhtin.
\newblock \emph{On t-entropy and variational principle for the spectral radius of weighted shift operators}.
\newblock Ergodic Theory Dynam. Systems 30 (2010) 1331--1342.

\bibitem{BaLe}
V. I. Bakhtin and A. V. Lebedev.
\newblock \emph{Entropy statistic theorem and variational principle for t-entropy are equivalent}.
\newblock J. Math. Anal. Appl. 474 (2019) 59--71.

\bibitem{BKM}
B. B\'ar\'any, A. K\"aenm\"aki and I. Morris.
\newblock \emph{Domination, almost additivity and thermodynamical formalism for planar matrix cocycles.}
\newblock Israel J. Math. 239 (2020) 173--214.

\bibitem{BR}
E. Barone and K.P.S. Bhaskara Rao.
\newblock \emph{Poincar\'e recurrence theorem for finitely additive measures.}
\newblock Rend. Mat. 7:1:4 (1981) 521--526.

\bibitem{Ba96}
L. Barreira.
\newblock \emph{A non-additive thermodynamic formalism and applications to dimension theory of hyperbolic dynamical systems.}
\newblock Ergodic Theory Dynam. Systems 16 (1996) 871--927.

\bibitem{Ba06}
L. Barreira.
\newblock \emph{Non-additive thermodynamic formalism: equilibrium and Gibbs measures.}
\newblock Discrete Contin. Dyn. Syst. 16 (2006) 279--305.

\bibitem{BB} K. P. S. Bhaskara and M. Bhaskara.
\newblock Theory of charges: a study of finitely additive measures.
\newblock Academic Press, London, 1983.

\bibitem{BCMV-2}
A. Bi\'s, M. Carvalho, M. Mendes and P. Varandas.
\newblock \emph{Entropy functions for semigroup actions}.
\newblock ArXiv:2009.07212



\bibitem{BiP}
E. Bishop and R. Phelps.
\newblock \emph{The support functionals of a convex set.}
\newblock Proc. Sympos. Pure Math. Vol. VII, Amer. Math. Soc., Providence R.I., 1963, pp. 27--35.

\bibitem{BM}
J. Bochi and I. Morris.
\newblock \emph{Equilibrium states of generalised singular value potentials and applications to affine iterated function systems}.
\newblock Geom. Funct. Analysis 28 (2018) 995--1028.

\bibitem{BR16}
J. Bochi and M. Rams.
\newblock \emph{The entropy of Lyapunov-optimizing measures of some matrix cocycles.}
\newblock  J. Modern Dynam. 10 (2016) 255--286.

\bibitem{BDV}
C. Bonatti, L. D\'iaz and M. Viana.
\newblock Dynamics Beyond Uniform Hyperbolicity.
\newblock Encyclopaedia of Mathematical Sciences 102, Springer-Verlag Berlin Heidelberg, 2004.

\bibitem{Bo}
N. Bourbaki.
\newblock General Topology. Chapters 5-10.
\newblock Elements of Mathematics, Springer-Verlag Berlin, 1998.

\bibitem{Bo75}
R.~Bowen.
\newblock Equilibrium States and the Ergodic Theory of Anosov Diffeomorphisms.
\newblock Lectures Notes in Math. 470, Springer-Verlag Berlin, 1975.

\bibitem{Bo-flows}
R. Bowen.
\newblock \emph{Symbolic dynamics for hyperbolic flows.}
\newblock Amer. J. Math. 95 (1973) 429--460.

\bibitem{BD}
M. Boyle and T. Downarowicz.
\newblock \emph{The entropy theory of symbolic extensions.}
\newblock Invent. Math. 156 (2004) 119--161.

\bibitem{BT}
H. Bruin and M. Todd.
\newblock \emph{Equilibrium states for potentials with $\sup\phi-\inf\phi < h_{\text{top}}(f)$}.
\newblock Comm. Math. Phys. 283 (2008) 579--611.

\bibitem{BTT}
H. Bruin, D. Terhesiu and M. Todd.
\newblock \emph{The pressure function for infinite equilibrium states.}
\newblock Israel J. Math. 232 (2019) 775--826.

\bibitem{Buf}
A. Bufetov.
\newblock \emph{Finitely-additive measures on the asymptotic foliations of a Markov compactum.}
\newblock Moscow Math. Journal 14:2 (2014) 205--224

\bibitem{Buz}
J. Buzzi.
\newblock \emph{Intrinsic ergodicity of smooth interval maps.}
\newblock Israel J. Math. 100 (1997) 125--161.

\bibitem{Bu14}
J. Buzzi.
\newblock \emph{$C^r$ surface diffeomorphisms with no maximal entropy measure}.
\newblock Ergodic Theory and Dynam. Systems 34:6 (2014) 1770--1793.

\bibitem{BCS}
J. Buzzi, S. Crovisier and O. Sarig.
\newblock \emph{Measures of maximal entropy for surface diffeomorphisms.}
\newblock ArXiv:1811.02240v2


\bibitem{CFH08}
Y. Cao, D. Feng and W. Huang.
\newblock \emph{The thermodynamic formalism for sub-additive potentials.}
\newblock Discrete Contin. Dyn. Syst. 20 (2008) 639--657.

\bibitem{CRV-Rec}
M. Carvalho, F. B. Rodrigues and P. Varandas.
\newblock \emph{Semigroup actions of expanding maps.}
\newblock J. Stat. Phys. 166 (2017) 114--136.




\bibitem{Chen1}
R. Chen.
\newblock \emph{A finitely additive version of Kolmogorov's law of iterated logarithm.}
\newblock Israel J. Math. 23 (1976) 209--220.

\bibitem{Chen2}
R. Chen.
\newblock \emph{Some finitely additive versions of the strong law of large numbers.}
\newblock Israel J. Math. 24 (1976) 244--259.

\bibitem{CES}
L. Cioletti, E. Silva and M. Stadlbauer.
\newblock \emph{Thermodynamic formalism for topological Markov chains on standard Borel spaces.}
\newblock Discrete Contin. Dyn. Syst. 39:11 (2019) 6277--6298.

\bibitem{CER}
L. Cioletti, A. van Enter and R. Ruviaro.
\newblock \emph{The double transpose of the Ruelle operator}.
\newblock ArXiv:1710.03841v2

%

\bibitem{CR19}
D. Coronel and J. Rivera-Letelier.
\newblock \emph{Sensitive dependence of geometric Gibbs states at positive temperature.}
\newblock Comm. Math. Phys. 368 (2019) 383--425.

\bibitem{Noe}
N. Cuneo.
\newblock \emph{Additive, almost additive and asymptotically additive potential sequences are equivalent.}
\newblock Comm. Math. Phys.  377 (2020) 2579--2595.

\bibitem{CKK}
M. Cuth, O. Kalenda and P. Kaplicky.
\newblock \emph{Finitely additive measures and complementability of Lipschitz-free spaces.}
\newblock Israel J. Math. 230:1 (2019) 409--442.

\bibitem{DE}
H. Dani{\"e}ls and A. van Enter.
\newblock \emph{Differentiability properties of the pressure in lattice systems.}
\newblock Comm. Math. Phys. 71 (1980) 65--76.

\bibitem{DN}
T. Downarowicz and S. Newhouse.
\newblock \emph{Symbolic extensions and smooth dynamical systems.}
\newblock Invent. Math. 160 (2005) 453--499.

\bibitem{D11}
T. Downarowicz.
\newblock Entropy in Dynamical Systems.
\newblock New Mathematical Monographs 18, Cambridge University Press, Cambridge, 2011.



\bibitem{Dub}
L. Dubins.
\newblock \emph{On Lebesgue-like extensions of finitely additive measures.}
\newblock  Ann. Probab. 2 (1974) 456--463.


\bibitem{DS}
N. Dunford and J. Schwartz.
\newblock Linear operators I
\newblock John Wiley \& Sons Inc., 1958.

\bibitem{DS-2}
N. Dunford and J. Schwartz.
\newblock Linear operators II.
\newblock Interscience Publishers Inc., 1963.

\bibitem{Ellis}
R. Ellis. Entropy, Large Deviations and Statistical Mechanics. Classics in Mathematics. Springer-Verlag Berlin, 2006.

%

\bibitem{FH}
D.-J. Feng and W. Huang.
\newblock \emph{Lyapunov spectrum of asymptotically sub-additive potentials.}
\newblock Comm. Math. Phys. 297 (2010) 1--43.

\bibitem{FK12}
D.-J. Feng and A. K\"aenm\"aki.
\newblock \emph{Equilibrium states of the pressure function for products of matrices.}
\newblock Discrete Cont. Dyn. Sys. 30 (2011) 699--708.

\bibitem{FK}
G. Fichtenholz and L. Kantorovich.
\newblock \emph{Sur les op\'erations lin\'eaires dans l'espace des fonctions born\'ees}.
\newblock Studia Math. 5 (1934) 69--98.

\bibitem{FS-B}
H. F\"ollmer and A. Schied.
\newblock Stochastic Finance, an Introduction in Discrete Time.
\newblock de Gruyter, Studies in Mathematics 27, Berlin, 2002.



\bibitem{GKLM}
P. Giulietti, B. Kloeckner, A. O. Lopes and D. Marcon.
\newblock \emph{The calculus of thermodynamical formalism.}
\newblock J. Eur. Math. Soc.  20 (2018) 2357--2412


\bibitem{HY}
E. Hewitt and K. Yosida.
\newblock \emph{Finitely additive measures.}
\newblock Trans. Amer. Math. Soc. 72 (1952) 46--66.

\bibitem{Hi}
T. Hildebrandt.
\newblock \emph{On bounded functional operations.}
\newblock Trans. Amer. Math. Soc. 36:4 (1934) 868--875.


\bibitem{IT}
G. Iommi and M. Todd.
\newblock \emph{Transience in dynamical systems.}
\newblock Ergodic Theory Dynam. Systems  33:5 (2013) 1450--1476.

\bibitem{Isr}
R. Israel.
\newblock \emph{Convexity in the Theory of Lattice Gases.}
\newblock Princeton Series in Physics, Princeton University Press, Princeton N.J., 1979. 

\bibitem{IP}
R. Israel and R. Phelps.
\newblock \emph{Some convexity questions arising in statistical mechanics.}
\newblock Math. Scand. 54 (1984) 133--156.

\bibitem{Je}
O. Jenkinson.
\newblock \emph{Ergodic optimization in dynamical systems.}
\newblock Ergodic Theory and Dynam. Systems 39:10 (2019) 2593--2618.

\bibitem{Kar}
R. Karandikar.
\newblock \emph{A general principle for limit theorems in finitely additive probability.}
\newblock Trans. Amer. Math. Soc. 273:2 (1982) 541--550.



\bibitem{Ke}
G. Keller.
\newblock \emph{Equilibrium States in Ergodic Theory}.
\newblock London Mathematical Society Student Texts, Cambridge University Press, Cambridge, 1998.


\bibitem{KT}
J. Kingman and S. Taylor.
\newblock Introduction to Measure and Probability.
\newblock Cambridge University Press, Cambridge, 1966.



\bibitem{Lo}
A. O. Lopes.
\newblock \emph{The zeta function, nondifferentiability of pressure, and the critical exponent of transition.}
\newblock Adv. Math. 101 (1993) 133--165.

\bibitem{Lopes}
A. O. Lopes, J. Mengue, J. Mohr and R. Souza.
\newblock \emph{Entropy and variational principle for one-dimensional lattice systems with a general a priori probability: positive and zero temperature.}
\newblock Ergodic Theory and Dynam. Systems 35:6 (2015) 1925--1961.

\bibitem{Mah}
D. Maharan.
\newblock \emph{Finitely additive measures on the integers}.
\newblock Sankhya: The Indian Journal of Statistics 38 (1976) 44--59.

\bibitem{May}
H. B. Maynard.
\newblock \emph{A Radon-Nikodym theorem for finitely additive bounded measures.}
\newblock Pacific J. Math. 83:2 (1979) 401--413.

\bibitem{Mazur}
S. Mazur.
\newblock \emph{{\"U}ber konvexe Mengen in linearen normierten R{\"a}umen.}
\newblock Studia Math. 4 (1933) 70--84.


\bibitem{Rz1}
R. Mohammadpour.
\newblock \emph{Zero temperature limits of equilibrium states for subadditive potentials and approximation of the maximal Lyapunov exponent.}
\newblock Topol. Methods Nonlinear Anal. 55:2 (2020) 697--710.


\bibitem{New89}
S. Newhouse.
\newblock \emph{Continuity properties of the entropy.}
\newblock Ann. of Math. 129 (1989) 215--237.


\bibitem{Nik}
O. Nikodym.
\newblock \emph{A theorem on infinite sequences of finitely additive real valued measures.}
\newblock Rend. Sem. Mat. Padova 24 (1955) 265--286.

\bibitem{Park}
K. Park.
\newblock \emph{Quasi-multiplicativity of typical cocycles.}
\newblock Comm. Math. Phys. 376:3 (2020) 1957--2004.

\bibitem{PP90}
W. Parry and M. Pollicott.
\newblock \emph{Zeta functions and the periodic orbit structure of hyperbolic dynamics.}
\newblock Ast\'erisque 187--188 (1990).


\bibitem{Pat}
A. Paterson.
\newblock Amenability.
\newblock Mathematical Surveys and Monographs, 29, American Mathematical Society, Providence R.I., 1988.


\bibitem{Phe93}
R. Phelps.
\newblock Convex Functions, Monotone Operators and Differentiability.
\newblock Lecture Notes in Math. 1364, Springer-Verlag New York, 2nd edition, 1993.

\bibitem{Pi}
V. Pinheiro.
\newblock \emph{Expanding measures.}
\newblock Ann. Inst. H. Poincar\'e Anal. Non Lin\'eaire 28 (2011) 889--939.


\bibitem{PRS}
F. Przytycki, J. Rivera-Letelier and S. Smirnov.
\newblock \emph{Equivalence and topological invariance of conditions for non-uniform hyperbolicity in the iteration of rational maps.}
\newblock Invent. Math. 151:1 (2003) 29--63.

\bibitem{PR}
F. Przytycki and J. Rivera-Letelier.
\newblock \emph{Geometric pressure for multimodal maps of the interval.}
\newblock Mem. Amer. Math. Soc. 259 (2019) 1246, v+81 pp.

\bibitem{PS}
R. Purves and W. Sudderth.
\newblock \emph{Some finitely additive probability.}
\newblock Ann. Probab. 4:2 (1976) 259--276.

\bibitem{RamCLT}
S. Ramakrishnan.
\newblock \emph{Central limit theorems in a finitely additive setting.}
\newblock Illinois J. Math. 28:1 (1984) 139--161.

\bibitem{RamBET}
S. Ramakrishnan.
\newblock \emph{A finitely additive generalization of Birkhoff's ergodic theorem}.
\newblock Proc. Amer. Math. Soc. 96:2 (1986) 299--305.

\bibitem{Ra}
M. Ratner.
\newblock \emph{Markov partitions for Anosov flows on n-dimensional manifolds.}
\newblock Israel J. Math. 15 (1973) 92--114.

\bibitem{RS}
M. Reed and B. Simon.
\newblock Methods of Modern Mathematical Physics: Functional Analysis.
\newblock Academic Press, London New York, 1981.

\bibitem{RV73}
A. Roberts and D. Varberg.
\newblock Convex Functions.
\newblock Academic Press, London New York, 1973.

\bibitem{Rohlin}
V. A. Rohlin.
\newblock  \emph{Exact endomorphisms of a Lebesgue space.}
\newblock Izv. Akad. Nauk SSSR Ser. Mat. 25 (1961) 499--530; English translation: Amer. Math. Soc. Transl. 39:2 (1964) 1--36.


\bibitem{Ru78}
D.~Ruelle.
\newblock Thermodynamic Formalism.
\newblock The Mathematical Structures of Classical Equilibrium Statistical Mechanics. Addison-Wesley, New York, 1978.

\bibitem{Rt2002}
S.~Ruette.
\newblock \emph{Mixing $C^r$ maps of the interval without maximal measure.}
\newblock Israel J. Math. 127 (2002) 253--277.

\bibitem{Sar13}
O. Sarig.
\newblock \emph{Symbolic dynamics for surface diffeomorphisms with positive entropy.}
\newblock J. Amer. Math. Soc. 26:2 (2013) 341--426.

\bibitem{Sar15}
O. Sarig.
\newblock \emph{Thermodynamic formalism for countable Markov shifts. Hyperbolic dynamics, fluctuations and large deviations}
\newblock Proc. Sympos. Pure Math. 89, Amer. Math. Soc., Providence R.I., 2015, 81--117.

\bibitem{Sc98}
S. Schreiber.
\newblock \emph{On growth rates of sub-additive functions for semi-flows.}
\newblock J. Diff. Equations 148 (1998) 334--350.

\bibitem{Si72}
Ya. Sinai.
\newblock \emph{Gibbs measures in ergodic theory.}
\newblock Russian Math. Surveys 27 (1972) 21--69.

\bibitem{Tol}
J. Toland.
\newblock The Dual of $L_\infty(X, \mathcal{L}, \lambda)$, Finitely Additive Measures and Weak Convergence.
\newblock Springer Briefs in Mathematics, Springer Nature Switzerland AG, 2020.

\bibitem{VV10}
P.~Varandas and M.~Viana.
\newblock \emph{Existence, uniqueness and stability of equilibrium states for non-uniformly expanding maps.}
\newblock Ann. Inst. H. Poincar\'e Anal. Non Lin\'eaire 27 (2010) 555--593.

\bibitem{Ve}
R. Venegeroles.
\newblock  \emph{Thermodynamic phase transitions for Pomeau-Manneville maps.}
\newblock  Physical Review E 86:2 (2012) 021114.

\bibitem{VY}
M.~Viana and J.~Yang.
\newblock Oral presentation at
\newblock \emph{International Conference on Dynamical Systems}, IMPA - Rio de Janeiro, 2015.

\bibitem{Wag}
S. Wagon.
\newblock The Banach-Tarski paradox.
\newblock Encyclopedia of Mathematics and its Applications, vol. 24, Cambridge University Press, Cambridge, 1985.

\bibitem{Wa}
P.~Walters.
\newblock An Introduction to Ergodic Theory.
\newblock Springer-Verlag New York, 1975.

\bibitem{Wa2}
P.~Walters.
\newblock \emph{Differentiability properties of the pressure of a continuous transformation on a compact metric space}.
\newblock J. Lond. Math. Soc. 46:3 (1992) 471--481.


\bibitem{Zub}
D. Zubov.
\newblock \emph{Finitely additive measures on the unstable leaves of Anosov diffeomorphisms.}
\newblock Functional Analysis and its Applications 53 (2019) 232--236.

\end{thebibliography}
\end{document}